\documentclass[reqno]{amsart}
\usepackage{amssymb}
\usepackage[usenames, dvipsnames]{color}
\usepackage{hyperref}

\usepackage{verbatim}

\numberwithin{equation}{section}

\newtheorem{theorem}{Theorem}[section]
\newtheorem{corollary}[theorem]{Corollary}
\newtheorem{lemma}[theorem]{Lemma}
\newtheorem{proposition}[theorem]{Proposition}

\theoremstyle{definition}
\newtheorem{remark}[theorem]{Remark}

\theoremstyle{definition}

\theoremstyle{definition}
\newtheorem{assumption}[theorem]{Assumption}

\makeatletter
\def\dashint{\operatorname%
{\,\,\text{\bf--}\kern-.98em\DOTSI\intop\ilimits@\!\!}}
\makeatother

\def\bH{\mathbb{H}}

\def\bN{\mathbb{N}}

\def\bR{\mathbb{R}}

\def\cA{\mathcal{A}}
\def\cB{\mathcal{B}}
\def\cC{\mathcal{C}}
\def\cD{\mathcal{D}}

\def\cG{\mathcal{G}}
\def\cH{\mathcal{H}}

\def\cM{\mathcal{M}}

\def\cS{\mathcal{S}}

\begin{document}
\title[time fractional parabolic equations]{$L_p$-estimates for time fractional parabolic equations with coefficients measurable in time}

\author[H. Dong]{Hongjie Dong}
\address[H. Dong]{Division of Applied Mathematics, Brown University, 182 George Street, Providence, RI 02912, USA}

\email{Hongjie\_Dong@brown.edu}

\thanks{H. Dong was partially supported by the NSF under agreement DMS-1600593.}

\author[D. Kim]{Doyoon Kim}
\address[D. Kim]{Department of Mathematics, Korea University, 145 Anam-ro, Seongbuk-gu, Seoul, 02841, Republic of Korea}

\email{doyoon\_kim@korea.ac.kr}

\thanks{D. Kim was supported by Basic Science Research Program through the National Research Foundation of Korea (NRF) funded by the Ministry of Education (2016R1D1A1B03934369).}



\begin{abstract}
We establish the $L_p$-solvability for time fractional parabolic equations when coefficients are merely measurable in the time variable.
In the spatial variables, the leading coefficients locally have small mean oscillations. Our results extend a recent result in \cite{MR3581300} to a large extent.
\end{abstract}

\maketitle

\section{Introduction}

In this paper, we consider time fractional parabolic equations with a non-local type time derivative term of the form
\begin{equation}
							\label{eq0525_01}
- \partial_t^\alpha u + a^{ij}(t,x) D_{ij} u + b^i(t,x) D_i u + c(t,x) u = f(t,x)
\end{equation}
in $(0,T) \times \bR^d$, where $\partial_t^\alpha u$ is the Caputo fractional derivative of order $\alpha \in (0,1)$:
$$
\partial_t^\alpha u(t,x) = \frac{1}{\Gamma(1-\alpha)} \frac{d}{dt} \int_0^t (t-s)^{-\alpha} \left[ u(s,x) - u(0,x) \right] \, ds.
$$
See Sections \ref{sec2} and \ref{Sec3} for a precise definition and properties of $\partial_t^\alpha u$.
Our main result is that, for a given $f \in L_p\left((0,T) \times \bR^d \right)$, there exists a unique solution $u$ to the equation \eqref{eq0525_01} in $(0,T) \times \bR^d$ with the estimate
$$
\||\partial_t^\alpha u|+|u|+|
Du|+|D^2u|\|_{L_p\left((0,T) \times \bR^d \right)} \le N \|f\|_{L_p\left((0,T) \times \bR^d \right)}.
$$

The assumptions on the coefficients $a^{ij}$, $b^i$, and $c$ are as follows.
The leading coefficients $a^{ij}=a^{ij}(t,x)$ satisfy the uniform ellipticity condition and have no regularity in the time variable.
Dealing with such coefficients in the setting of $L_p$ spaces is the main focus of this paper.
As functions of $x$, locally the coefficients $a^{ij}$ have small (bounded) mean oscillations (small BMO).
See Assumption \ref{assump2.2}.
The lower-order coefficients $b^i$ and $c$ are assumed to be only bounded and measurable.

If the fractional (or non-local) time derivative $\partial_t^\alpha u$ is replaced with the local time derivative $u_t$, the equation \eqref{eq0525_01} becomes the usual second-order non-divergence form parabolic equation
\begin{equation}
							\label{eq0525_02}
-u_t + a^{ij} D_{ij} u + b^i D_i u + c u = f.
\end{equation}
As is well known, there is a great amount of literature on the regularity and solvability for equations as in \eqref{eq0525_02} in various function spaces.
Among them, we only refer the reader to the papers \cite{MR2304157, MR2352490, MR2771670}, which contain corresponding results of this paper to parabolic  equations as in \eqref{eq0525_02}.
More precisely, in these papers, the unique solvability results are proved in Sobolev spaces for elliptic and parabolic equations/systems. In particular, for the parabolic case, the leading coefficients are assumed to satisfy the same conditions as mentioned above.
This class of coefficients was first introduced by Krylov in \cite{MR2304157} for parabolic equations in Sobolev spaces.
In \cite{MR2352490}, the results in \cite{MR2304157} were generalized to the mixed Sobolev norm setting, and in \cite{MR2771670} to higher-order elliptic and parabolic systems.
Thus, one can say that the unique solvability of solutions in  Sobolev spaces to parabolic equations as in \eqref{eq0525_02} is well established when coefficients are merely measurable in the time variable.
On the other hand, it is well known that the $L_p$-solvability of elliptic and parabolic equations requires the leading coefficients to have some regularity conditions in the spatial variables.
See, for instance, the paper \cite{MR3488249}, where the author shows the impossibility of finding solutions in $L_p$ spaces to one spatial dimensional parabolic equations if $p \notin (3/2,3)$ and the leading coefficient are merely measurable in $(t,x)$.

In view of mathematical interests and applications, it is a natural and interesting question to explore whether the corresponding $L_p$-solvability results hold for equations as in \eqref{eq0525_01} for the same class of coefficients as in \cite{MR2304157, MR2352490, MR2771670}.
In a recent paper \cite{MR3581300} the authors proved the unique solvability of solutions in mixed $L_{p,q}$ spaces to the time fractional parabolic equation  \eqref{eq0525_01} under the stronger assumption that the leading coefficients are piecewise continuous in time and uniformly continuous in the spatial variables.
Hence, the results in this paper can be regarded as a generalization of the results in \cite{MR3581300} to a large extent, so that one can have the same class of coefficients as in \cite{MR2304157, MR2352490, MR2771670} for the time non-local equation \eqref{eq0525_01} in $L_p$ spaces.
We note that in \cite{MR3581300} the authors discussed the case $\alpha \in (0,2)$, whereas in this paper we only discuss the parabolic regime $\alpha \in (0,1)$.
It is also worth noting that, for parabolic equations as in \eqref{eq0525_02}, it is possible to consider more general classes of coefficients than those in \cite{MR2304157, MR2352490, MR2771670}.
Regarding this, see \cite{DK15}, where the classes of coefficients under consideration include those $a^{ij}(t,x)$ measurable both in one spatial direction and in time except, for instance, $a^{11}(t,x)$, which is measurable either in time or in the spatial direction.

Besides \cite{MR3581300}, there are a number of papers about parabolic equations with a non-local type time derivative term.
For divergence type time fractional parabolic equations in the Hilbert space setting, see \cite{MR2538276}, where the time fractional derivative is a generalized version of the Caputo fractional derivative.
One can find De Giorgi-Nash-Moser type H\"{o}lder estimates for time fractional parabolic equations in \cite{MR3038123}, and for parabolic equations with fractional operators in both $t$ and $x$ in \cite{MR3488533}.
For other related papers and further information about time fractional parabolic equations and their applications, we refer to \cite{MR3581300} and the references therein.

As a standard scheme in $L_p$-theory, to establish the main results of this paper, we prove a priori estimates for solutions to \eqref{eq0525_01}.
In \cite{MR3581300} a representation formula for a solution to the time fractional heat operator $-\partial_t^\alpha u + \Delta u$ is used, from which the $L_p$-estimate is derived for the operator.
Then for uniformly continuous coefficients, a perturbation argument takes places to derive the main results of the paper.
Our proof is completely different.
Since $a^{ij}$ are measurable in time, it is impossible to treat the equation via a perturbation argument from the time fractional heat equation.
Thus, instead of considering a representation formula for equations with coefficients measurable in time, which does not seem to be available, we start with the $L_2$-estimate and solvability, which can be obtained from integration by parts.
We then exploit a level set argument originally due to Caffarelli and Peral \cite{MR1486629} as well as a ``crawling of ink spots'' lemma, which was originally due to Safonov and Krylov \cite{MR579490, MR563790}.
The main difficulty arises in the key step where one needs to estimate local $L_\infty$ estimates of the Hessian of solutions to locally homogeneous equations.
Starting from the $L_2$-estimate and applying the Sobolev type embedding results proved in Appendix, we are only able to show that such Hessian are in $L_{p_1}$ for some $p_1>2$, instead of $L_\infty$.
Nevertheless, this allows us to obtain the $L_p$ estimate and solvability for any $p\in [2,p_1)$ and $a^{ij}=a^{ij}(t)$ by using a modified level set type argument.
Then we repeat this procedure and iteratively increase the exponent $p$ for any $p\in [2,\infty)$. In the case when $p\in (1,2)$, we apply a duality argument.
For equations with the leading coefficients being measurable in $t$ and locally having small mean oscillations in $x$, we apply a perturbation argument (see, for instance, \cite{MR2304157}).
This is done by incorporating the small mean oscillations of the coefficients into local mean oscillation estimates of solutions having compact support in the spatial variables.
Then, the standard partition of unity argument completes the proof.

In forthcoming work, we will generalize our results for time fraction parabolic equations with more general coefficients considered, for example, in \cite{DK15}.  We will also consider solutions in Sobolev spaces with mixed norms as in \cite{MR3581300} as well as equations in domains.

The remainder of the paper is organized as follows.
In the next section, we introduce some notation and state the main results of the paper. In Section \ref{Sec3}, we define function spaces for fractional time derivatives and show some of their properties. In Section \ref{sec4}, we prove the $L_2$ estimate and solvability for equations with coefficients depending only on $t$, and then derive certain local estimates, which will be used later in the iteration argument. We give the estimates of level sets of Hessian in Section \ref{sec5} and complete the proofs of the main theorems in Section \ref{sec6}. In Appendix, we establish several Sobolev type embedding theorems involving time fractional derivatives and prove a ``crawling of ink spots'' lemma adapted to our setting.

\section{Notation and main results}
							\label{sec2}

We first introduce some notation used through the paper.
For $\alpha \in (0,1)$, denote
$$
I^\alpha \varphi(t) = I_0^\alpha \varphi(t) = \frac{1}{\Gamma(\alpha)} \int_0^t (t-s)^{\alpha - 1} \varphi(s) \, ds
$$
for $\varphi \in L_1(\bR^+)$,
where
$$
\Gamma(\alpha) = \int_0^\infty t^{\alpha - 1} e^{-t} \, dt.
$$
In \cite{MR1544927} $I^\alpha \varphi$ is called $\alpha$-th integral of $\varphi$ with origin $0$.
For $0 < \alpha < 1$ and sufficiently smooth function $\varphi(t)$, we set
$$
D_t^\alpha \varphi(t) = \frac{d}{dt} I^{1-\alpha} \varphi(t) = \frac{1}{\Gamma(1-\alpha)} \frac{d}{dt} \int_0^t (t-s)^{-\alpha} \varphi(s) \, ds,
$$
and
\begin{align*}
\partial_t^\alpha \varphi(t) &= \frac{1}{\Gamma(1-\alpha)} \int_0^t (t-s)^{-\alpha} \varphi'(s) \, ds\\
&= \frac{1}{\Gamma(1-\alpha)} \frac{d}{dt} \int_0^t (t-s)^{-\alpha} \left[ \varphi(s) - \varphi(0) \right] \, ds.
\end{align*}
Note that if $\varphi(0) = 0$, then
$$
\partial_t(I^{1-\alpha} \varphi) = \partial_t^\alpha \varphi.
$$
Let $\cD$ be a subset (not necessarily open) of $\bR^k$, $k \in \{1,2, \ldots\}$.
By $\varphi \in C_0^\infty(\cD)$,
we mean that $\varphi$ is infinitely differentiable in $\cD$ and is supported in the intersection of $\cD$ and a bounded open subset in $\bR^d$.
In particular, $\varphi$ may not be zero on the boundary of $\cD$, unless $\cD$ is an open subset of $\bR^k$.
For $\alpha \in (0,1)$, we denote
$$
Q_{R_1,R_2}(t,x) = (t-R_1^{2/\alpha}, t) \times B_{R_2}(x) \quad \text{and} \quad Q_R(t,x)=Q_{R,R}(t,x).
$$
We often write $B_R$ and $Q_R$ instead of $B_R(0)$ and $Q_R(0,0)$, respectively.

In this paper, we assume that there exists $\delta \in (0,1)$ such that
$$
a^{ij}(t,x)\xi_j \xi_j \geq \delta |\xi|^2,\quad |a^{ij}| \leq \delta^{-1}
$$
for any $\xi \in \bR^d$ and $(t,x) \in \bR \times \bR^d$.

Our first main result is for equations with coefficients $a^{ij}$ depending only on the time variable without any regularity assumptions.

\begin{theorem}
							\label{thm0412_1}
Let $\alpha \in (0,1)$, $T \in (0,\infty)$, $a^{ij} = a^{ij}(t)$, and $p \in (1,\infty)$.
Suppose that $u \in \bH_{p,0}^{\alpha,2}(\bR^d_T)$ satisfies
\begin{equation}
							\label{eq0411_03}
-\partial_t^\alpha u + a^{ij} D_{ij} u
= f
\end{equation}
in $\bR^d_T := (0,T) \times \bR^d$.
Then there exists $N = N(d,\delta,\alpha,p)$ such that
\begin{equation}
							\label{eq0411_04}
\|\partial_t^\alpha u\|_{L_p(\bR^d_T)} + \|D^2 u\|_{L_p(\bR^d_T)} \leq N \|f\|_{L_p(\bR^d_T)}.
\end{equation}
Moreover, for $f \in L_p(\bR^d_T)$, there exists a unique $u \in \bH_{p,0}^{\alpha,2}(\bR^d_T)$ satisfying \eqref{eq0411_03} and \eqref{eq0411_04}.
\end{theorem}

We refer the reader to Section \ref{Sec3} for the definitions of function spaces including $\bH_{p,0}^{\alpha,2}(\bR^d_T)$.

We also consider more general operators with lower-order terms and with coefficients depending on both $t$ and $x$. In this case, we impose the following VMO$_x$ condition on the leading coefficients.
\begin{assumption}[$\gamma_0$]
                        \label{assump2.2}
There is a constant $R_0\in (0,1]$ such that for each parabolic cylinder $Q_r(t_0,x_0)$ with $r\le R_0$ and $(t_0,x_0)\in \bR^{d+1}$, we have
$$
\sup_{i,j}\dashint_{Q_r(t_0,x_0)}|a^{ij}-\bar a^{ij}(t)|\,dx\,dt\le \gamma_0,
$$
where $\bar a^{ij}(t)$ is the average of $a^{ij}(t,\cdot)$ in $B_r(x_0)$.
\end{assumption}

\begin{remark}
                            \label{rem2.3}
From the above assumption, we have that for any $x_0\in \bR^d$ and $a, b \in \bR$ such that $b - a > R_0^{2/\alpha}$, there exists $\bar{a}^{ij}(t)$ satisfying the ellipticity condition and
$$
\dashint_{\!a}^{\,\,\,b} \dashint_{B_{R_0}(x_0)} |a^{ij} - \bar{a}^{ij}(t)| \, dx \, dt \leq 2 \gamma_0.
$$
Indeed, find $k \in \{1,2, \ldots\}$ such that
$$
b - (k+1) R_0^{2/\alpha} \leq a < b - k R_0^{2/\alpha},\quad \text{i.e.,}\,\,
\frac{1}{k+1} \leq \frac{R_0^{2/\alpha}}{b-a} < \frac{1}{k},
$$
and set $\bar a^{ij}(t)$ to be the average of $a^{ij}(t,\cdot)$ in $B_{R_0}(x_0)$.
We then see that
\begin{align*}
&\dashint_{\!a}^{\,\,\,b} \dashint_{B_{R_0}(x_0)} |a^{ij} - \bar{a}^{ij}(t)| \, dx \, dt = \frac{1}{b-a} \int_a^b \dashint_{B_{R_0}(x_0)} |a^{ij} - \bar{a}^{ij}(t)| \, dx \, dt\\
&\leq \frac{R_0^{2/\alpha}}{b-a} \sum_{j=0}^k \dashint_{\!b-(j+1)R_0^{2/\alpha}}^{\,\,\,b-j R_0^{2/\alpha}} \dashint_{B_{R_0}(x_0)} |a^{ij} - \bar{a}^{ij}(t)| \, dx \, dt\\
&\leq \frac{R_0^{2/\alpha}}{b-a} (k+1) \gamma_0 \leq \frac{k+1}{k} \gamma_0 \leq 2 \gamma_0.
\end{align*}
\end{remark}

We also assume that the lower-order coefficients $b^i$ and $c$ satisfy
$$
|b^i|\le \delta^{-1},\quad |c|\le \delta^{-1}.
$$
\begin{theorem}
							\label{main_thm}
Let $\alpha \in (0,1)$, $T \in (0,\infty)$, and $p \in (1,\infty)$. There exists $\gamma_0\in (0,1)$ depending only on $d$, $\delta$, $\alpha$, and $p$, such that, under Assumption \ref{assump2.2} ($\gamma_0$), the following hold. Suppose that $u \in \bH_{p,0}^{\alpha,2}(\bR^d_T)$ satisfies
\begin{equation}
							\label{eq0411_03c}
-\partial_t^\alpha u + a^{ij} D_{ij} u+b^i D_i u+cu
= f
\end{equation}
in $\bR^d_T$.
Then there exists $N = N(d,\delta,\alpha,p,R_0,T)$ such that
\begin{equation}
							\label{eq0411_04c}
\|u\|_{\bH_p^{\alpha,2}(\bR^d_T)} \leq N \|f\|_{L_{p}(\bR^d_T)}.
\end{equation}
Moreover, for $f \in L_{p}(\bR^d_T)$, there exists a unique $u \in \bH_{p,0}^{\alpha,2}(\bR^d_T)$ satisfying \eqref{eq0411_03c} and \eqref{eq0411_04c}.
\end{theorem}

\section{Function spaces}
                        \label{Sec3}

Let $\Omega$ be a domain (open and connected, but not necessarily bounded) in $\bR^d$.
For $T > 0$, we denote
$$
\Omega_T = (0,T) \times \Omega \subset \bR \times \bR^d.
$$
Thus, if $\Omega = \bR^d$, we write $\bR^d_T = (0,T) \times \bR^d$.

For $S>-\infty$ and $\alpha\in (0,1)$, let $I_S^{1-\alpha} u$ be the $(1-\alpha)$-th integral of $u$ with origin $S$:
$$
I_S^{1-\alpha} u = \frac{1}{\Gamma(1-\alpha)}\int_S^t (t-s)^{-\alpha} u(s, x) \, ds.
$$
Throughout the paper, $I_0^{1-\alpha}$ is denoted by $I^{1-\alpha}$.

For $1 \le p \le \infty$, $\alpha \in (0,1)$, $T > 0$, and $k \in \{1,2,\ldots\}$, we set
$$
\widetilde{\bH}_p^{\alpha,k}(\Omega_T) = \left\{ u \in L_p(\Omega_T): D_t^\alpha u, \, D^\beta_x u \in L_p(\Omega_T), \, 0 \leq |\beta| \leq k
\right\}
$$
with the norm
\begin{equation*}
\|u\|_{\widetilde{\bH}_p^{\alpha,k}(\Omega_T)} = \|D_t^\alpha u\|_{L_p(\Omega_T)} + \sum_{0 \leq |\beta| \leq k}\|D_x^\beta u\|_{L_p(\Omega_T)},
\end{equation*}
where by $D_t^\alpha u$ or $\partial_t(I^{1-\alpha}u) (= \partial_t (I_0^{1-\alpha}u))$ we mean that there exists $g \in L_p(\Omega_T)$ such that
\begin{equation}
							\label{eq0122_01}
\int_0^T\int_\Omega g(t,x) \varphi(t,x) \, dx \, dt = - \int_0^T\int_\Omega I^{1-\alpha}u(t,x) \partial_t \varphi(t,x) \, dx \, dt
\end{equation}
for all $\varphi \in C_0^\infty(\Omega_T)$.
If we have a domain $(S,T) \times \Omega$ in place of $\Omega_T$, where $-\infty < S < T < \infty$, we write
$\widetilde{\bH}_p^{\alpha,k}\left((S,T) \times \Omega\right)$.
In this case
$$
D_t^\alpha u(t,x) = \partial_t I_S^{1-\alpha} u(t,x).
$$

Now we set
$$
\bH_p^{\alpha,k}(\Omega_T) = \left\{ u \in \widetilde{\bH}_p^{\alpha,k}(\Omega_T): \text{\eqref{eq0122_01} is satisfied for all}\,\, \varphi \in C_0^\infty\left([0,T) \times \Omega\right)\right\}
$$
with the same norm as for $\widetilde{\bH}_p^{\alpha,k}(\Omega_T)$.
Similarly, we define $\bH_p^{\alpha,k}((S,T)\times\Omega)$.
If \eqref{eq0122_01} holds for all functions $\varphi \in C_0^\infty\left([0,T) \times \Omega \right)$, then one can regard that $I^{1-\alpha}u(t)|_{t=0} = 0$ in the trace sense with respect to the time variable.
In Lemma \ref{lem0123_1} below, we  show that, if $\alpha \leq 1 - 1/p$, then $\bH_p^{\alpha,k}(\Omega_T)=\widetilde{\bH}_p^{\alpha,k}(\Omega_T)$.

\begin{lemma}
							\label{lem0123_1}
Let $p \in [1,\infty]$, $\alpha \in (0,1)$, $k \in \{1,2,\ldots\}$ and
$$
\alpha \le 1 - 1/p.
$$
Then, for $u \in \widetilde{\bH}_p^{\alpha,k}(\Omega_T)$, the equality \eqref{eq0122_01} holds for all $\varphi \in C_0^\infty\left([0,T) \times \Omega\right)$.
\end{lemma}

\begin{proof}
Let $\eta_k(t)$ be an infinitely differentiable function such that $0 \leq \eta_k(t) \leq 1$, $\eta_k(t) = 0$ for $t \leq 0$, $\eta_k(t) = 1$ for $t \geq 1/k$, and $|\partial_t \eta_k(t)| \leq 2k$.
Then
\begin{align*}
&\int_0^T I^{1-\alpha} u(t,x) \partial_t (\varphi(t) \eta_k(t)) \, dt\\
&= \int_0^T I^{1-\alpha}u(t,x) \partial_t \varphi(t) \eta_k(t) \, dt + \int_0^T  I^{1-\alpha} u(t,x) \varphi(t) \partial_t \eta_k(t) \, dt.
\end{align*}
To prove the desired equality, we only need to show that
$$
\int_0^T \int_\Omega I^{1-\alpha}u(t,x) \varphi(t) \partial_t \eta_k(t)\, dx \, dt \to 0
$$
as $k \to \infty$.
Note that
$$
\int_0^T I^{1-\alpha}u(t,x) \varphi(t) \partial_t \eta_k(t) \, dt = \int_0^{1/k} I^{1-\alpha}u(t,x) \varphi(t) \partial_t \eta_k(t) \, dt =:J_k(x).
$$
Then, by Lemma \ref{lem1018_01} with $1-\alpha$ in place of $\alpha$, for any $q \in [1,\infty]$ satisfying
$$
1- \alpha - 1/p > -1/q,
$$
we have
\begin{align*}
|J_k(x)| &\leq N k \int_0^{1/k} I^{1-\alpha} |u(t,x)| \, dt
\leq N k^{1/q} \left( \int_{0}^{1/k} \left|I^{1-\alpha}|u(t,x)|\right|^q \, dt \right)^{1/q}\\
&\leq N k^{\alpha - 1 + 1/p}\|u(\cdot,x)\|_{L_p(0,1/k)} \to 0
\end{align*}
as $k \to \infty$, provided that $\alpha \le 1 - 1/p$.
The lemma is proved.
\end{proof}

We now prove that every function in $\bH_p^{\alpha,k}
(\Omega_T)$ can be approximated by infinitely differentiable functions up to the boundary with respect to the time variable.

\begin{proposition}
							\label{prop0120_1}
Let $p \in [1,\infty)$, $\alpha \in (0,1)$, and $k \in \{1,2,\ldots\}$.
Then functions in $C^\infty\left([0,T] \times \Omega\right)$ vanishing for large $|x|$ are dense in $\bH_p^{\alpha,k}(\Omega_T)$.
\end{proposition}

\begin{proof}
We prove only the case when $\Omega = \bR^d$.
More precisely, we show that
$C^\infty_0\left([0,T] \times \bR^d\right)$ is dense in $\bH_p^{\alpha,k}(\bR^d_T)$.
The proof of the case when $\Omega = \bR^d_+$ is similar.
For a general $\Omega$, the claim is proved using a partition of unity with respect to the spatial variables.
See, for instance, \cite{MR0164252}.

Let $u \in \bH_p^{\alpha,k}(\bR^d_T)$.
Let $\eta(t,x)$ be an infinitely differentiable function defined in $\bR^{d+1}$ satisfying $\eta \ge 0$,
$$
\eta(t,x) = 0 \quad \text{outside} \,\,(0,1)\times B_1,\quad \int_{\bR^{d+1}} \eta \, dx \, dt = 1.
$$
Set
$$
\eta_\varepsilon(t,x) = \frac{1}{\varepsilon^{d+2/\alpha}} \eta(t/\varepsilon^{2/\alpha}, x/\varepsilon)
$$
and
$$
u^{(\varepsilon)}(t,x) = \int_\bR \int_{\bR^d} \eta_{\varepsilon}(t-s,x-y) u(s,y) I_{0 < s < T} \, dy \, ds.
$$
Then it follows easily that $u^{(\varepsilon)}(t,x) \in C^\infty(\bR^{d+1})$ and, for $(t,x) \in (0,T) \times \bR^d$ and $0 \leq |\beta| \leq k$,
\begin{equation}
							\label{eq0120_01}
D^\beta_x u^{(\varepsilon)}(t,x) = \int_\bR \int_{\bR^d} \eta_{\varepsilon}(t-s,x-y) D^\beta_x u(s,y)  I_{0<s<T} \, dy \, ds.
\end{equation}
Moreover, for $(t,x) \in (0,T) \times \bR^d$,
\begin{equation}
							\label{eq0120_02}
D_t^\alpha u^{(\varepsilon)}(t,x) = \int_{\bR}\int_{\bR^d} \eta_{\varepsilon}(t-s,x-y) D^\alpha_t u(s,y)  I_{0<s<T} \, dy \, ds.
\end{equation}
To see \eqref{eq0120_02}, we first check that
\begin{equation}
							\label{eq0124_01}
I^{1-\alpha} u^{(\varepsilon)}(t,x) = (I^{1-\alpha} u)^{(\varepsilon)}(t,x).
\end{equation}
Indeed,
\begin{align*}
&\Gamma(1-\alpha) I^{1-\alpha} u^{(\varepsilon)}(t,x)\\
&= \int_0^t (t-s)^{-\alpha}\int_0^T \int_{\bR^d} \eta_\varepsilon(s-r,x-y) u(r,y) \, dy \, dr \, ds\\
&= \int_{\bR^d} \int_0^T \int_0^t (t-s)^{-\alpha} \eta_\varepsilon(s-r,x-y) u(r,y) \, ds \, dr \, dy\\
&= \int_{\bR^d} \int_0^t \int_r^t (t-s)^{-\alpha} \eta_\varepsilon(s-r,x-y) u(r,y) \, ds \, dr \, dy,
\end{align*}
where we used the fact that $\eta(t,x) = 0$ if $t \leq 0$.
Then by the change of variable $\rho = t-s+r$ in the integration with respect to $s$,
we have
\begin{align*}
&\Gamma(1-\alpha) I^{1-\alpha} u^{(\varepsilon)}(t,x) = \int_{\bR^d} \int_0^t \int_r^t (\rho-r)^{-\alpha} \eta_\varepsilon(t-\rho,x-y) u(r,y) \, d\rho \, dr \, dy\\
&= \int_{\bR^d} \int_0^t \eta_\varepsilon(t-\rho,x-y) \int_0^\rho (\rho-r)^{-\alpha} u(r,y)  \, dr \, d\rho \, dy\\
&= \int_{\bR^d} \int_0^T \eta_\varepsilon(t-\rho,x-y) \int_0^\rho (\rho-r)^{-\alpha} u(r,y)  \, dr \, d\rho \, dy\\
&= \Gamma(1-\alpha) (I^{1-\alpha}u)^{(\varepsilon)}(t,x).
\end{align*}
Hence, the inequality \eqref{eq0124_01} is proved.

Now observe that
\begin{align*}
&\int_{\bR}\int_{\bR^d} \eta_{\varepsilon}(t-s,x-y) D^\alpha_t u(s,y)  I_{0<s<T} \, dy \, ds\\
&= \int_0^T \int_{\bR^d} \eta_\varepsilon(t-s,x-y) \partial_s I^{1-\alpha}u(s,y) \, dy \, ds\\
&= \int_0^T \int_{\bR^d} (\partial_t \eta_\varepsilon) (t-s,x-y) I^{1-\alpha}u(s,y) \, dy \, ds\\
&= \partial_t \left[\int_0^T \int_{\bR^d} \eta_\varepsilon(t-s,x-y) I^{1-\alpha} u(s,y) \, dy \, ds\right]\\
&= \partial_t (I^{1-\alpha} u)^{(\varepsilon)}(t,x) = \partial_t I^{1-\alpha} u^{(\varepsilon)}(t,x) = D_t^\alpha u^{(\varepsilon)}(t,x),
\end{align*}
where in the second equality we used the fact that $u$ satisfies \eqref{eq0122_01} for all $\varphi \in C_0^\infty\left([0,T) \times \bR^d\right)$
and, by the choice of $\eta$, $\eta_\varepsilon(t-T,x-y) = 0$.
From the equalities \eqref{eq0120_01} and \eqref{eq0120_02}, we see that
$$
\|u^{(\varepsilon)} - u\|_{\bH_p^{\alpha,k}(\bR^d_T)} \to 0
$$
as $\varepsilon \to 0$.
Finally, we take a smooth cutoff function $\zeta\in C_0^\infty(\bR^d)$ such that $\operatorname{supp} \zeta \subset B_2$ and $\zeta=1$ in $B_1$, and denote $\zeta_\varepsilon(x)=\zeta(x/\varepsilon)$. Then by the uniform bound of $
\|u^{(\varepsilon)}\|_{\bH_p^{\alpha,k}(\bR^d_T)}$, it is easily seen that
$$
\|u^{(\varepsilon)} - u^{(\varepsilon)}\zeta_\varepsilon\|_{\bH_p^{\alpha,k}(\bR^d_T)} \to 0
$$
as $\varepsilon \to 0$. The lemma is proved.
\end{proof}

\begin{remark}
							\label{rem0606_1}
If the boundary of $\Omega$ is sufficiently smooth, for instance $\Omega$ is a Lipschitz domain, then $C^\infty\big([0,T] \times \overline{\Omega} \big)$ is dense in $\bH_p^{\alpha,k}(\Omega_T)$.
\end{remark}

\begin{remark}
Lemma \ref{lem0123_1} shows that $\bH_p^{\alpha,k}(\Omega_T) = \widetilde{\bH}_p^{\alpha,k}(\Omega_T)$ whenever $\alpha \leq 1 - 1/p$, $p \in [1,\infty]$.
Hence, by Proposition \ref{prop0120_1}, it follows that functions in $C^\infty\left([0,T] \times \Omega\right)$ vanishing for large $|x|$ are dense in $\widetilde{\bH}_p^{\alpha,k}(\Omega_T)$, provided that $\alpha \leq 1 - 1/p$, $p \in [1,\infty)$, $\alpha \in (0,1)$, and $k \in \{1,2,\dots\}$.
However, in the case $\alpha > 1 - 1/p$, we have $$
\bH_p^{\alpha,k}(\Omega_T) \subsetneq \widetilde{\bH}_p^{\alpha,k}(\Omega_T).
$$
To see this, let
$$
u(t)=t^{\alpha-1},
$$
where $\alpha \in (1-1/p,1)$ and $p \in [1,\infty)$.
Then $u \in L_p(0,T)$ and
$$
I^{1-\alpha} u(t) = \frac{1}{\Gamma(1-\alpha)} \int_0^t (t-s)^{-\alpha} s^{\alpha-1} \, ds = \Gamma(\alpha),
$$
which is a nonzero constant, so that $$
\partial_t I^{1-\alpha} u=0.
$$
Thus,
$$
u, \, D_t^\alpha u \in L_p(0,T).
$$
However, clearly
the integration by parts formula \eqref{eq0122_01} does not hold for $\varphi \in C_0^\infty[0,T)$.
The above example also shows that, even though we have
$$
u, \, D_t^\alpha u\in L_p((0,T))
$$
for $\alpha > 1-1/p$, it is not likely to gain better integrability or regularity (up to the boundary) of $u$, as apposed to the usual Sobolev embedding results.
\end{remark}

To deal with solutions with the zero initial condition,
we define
$\bH_{p,0}^{\alpha,k}((S,T)\times\Omega)$
to be functions in $\bH_p^{\alpha,k}((S,T)\times\Omega)$
each of which is
approximated by a sequence $\{u_n(t,x)\} \subset C^\infty\left([S,T]\times \Omega\right)$ such that $u_n$ vanishes for large $|x|$ and $u_n(S,x) = 0$.
For $u \in \bH_{p,0}^{\alpha,k}((S,T)\times\Omega)$ and for any approximation sequences $\{u_n\}$ such that
$u_n \to u$ in $\bH_p^{\alpha,k}
((S,T)\times\Omega)$ with $u_n \in C^\infty\left([S,T] \times \Omega\right)$ and $u_n(S,x) = 0$, we have
$$
\partial_t^\alpha u_n = D_t^\alpha u_n.
$$
Thus, when, for instance, $S=0$, for $u \in \bH_{p,0}^{\alpha,k}(\Omega_T)$, we define
$$
\partial_t^\alpha u := D_t^\alpha u = \frac{1}{\Gamma(1-\alpha)} \partial_t \int_0^t (t-s)^{-\alpha} u(s,x) \, ds.
$$

\begin{lemma}
							\label{lem0206_1}
Let $p \in [1,\infty)$, $\alpha \in (0,1)$, $k \in \{1,2,\ldots\}$, $-\infty < S < t_0 < T < \infty$, and $u \in \bH_{p,0}^{\alpha,k}\left( (t_0,T) \times \Omega \right)$.
If $u$ is extended to be zero for $t \leq t_0$, denoted by $\bar{u}$, then $\bar{u} \in \bH_{p,0}^{\alpha,k}\left((S,T) \times \Omega\right)$.
\end{lemma}

\begin{proof}
Without loss of generality, we assume $t_0 = 0$ so that
$$
- \infty < S < 0 < T < \infty.
$$
For $u \in \bH_{p,0}^{\alpha,k}(\Omega_T)$,
let $\{u_n\}$ be an approximating sequence of $u$ such that $u_n \in \bH_{p,0}^{\alpha,k}(\Omega_T) \cap C^\infty\left([0,T] \times \Omega\right)$, $u_n$ vanishes for large $|x|$, and $u_n(0,x) = 0$.
Extend $u_n$ to be zero for $t \leq 0$, denoted by $\bar{u}_n$.
It is readily seen that, for $0 \leq |\beta| \leq k$,
$$
D_x^\beta \bar{u}_n = \left\{
\begin{aligned}
D_x^\beta u_n, \quad 0 \leq t \leq T,
\\
0, \quad S \leq t < 0,
\end{aligned}
\right.
$$
$$
D^\beta_x \bar{u}_n \in L_p\left((S,T) \times \Omega \right).
$$
Now we check that
\begin{equation}
							\label{eq0120_03}
D_t^\alpha \bar{u}_n = \partial_t I_S^{1-\alpha} \bar{u}_n =
\left\{
\begin{aligned}
\partial_t I_0^{1-\alpha} u_n, \quad 0 \leq t \leq T,
\\
0, \quad S \leq t < 0,
\end{aligned}
\right.
\end{equation}
$$
D_t^\alpha \bar{u}_n \in L_p\left((S,T) \times \Omega\right).
$$
To see this, note that $I_S^{1-\alpha}\bar{u}_n(t,x) = 0$ for $S \leq t < 0$.
For $0 \leq t \leq T$, we have
\begin{align*}
&I_S^{1-\alpha} \bar{u}_n = \frac{1}{\Gamma(1-\alpha)} \int_S^t (t-s)^{-\alpha} \bar{u}_n(s,x) \, dy\\
&= \frac{1}{\Gamma(1-\alpha)} \int_0^t (t-s)^{-\alpha} u_n(s,x) \, dy
= I_0^{1-\alpha} u_n(t,x).
\end{align*}
We now observe that, for $\varphi \in C_0^\infty\left( (S,T) \times \Omega\right)$,
\begin{align*}
&\int_S^T \int_\Omega I_S^{1-\alpha} \bar{u}_n(t,x) \varphi_t(t,x) \, dx \, dt = \int_0^T \int_\Omega I_0^{1-\alpha} u_n (t,x) \varphi_t(t,x) \, dx \, dt\\
&= - \int_0^T \int_\Omega \partial_t I_0^{1-\alpha} u_n(t,x) \varphi(t,x) \, dx \, dt,
\end{align*}
where we used the fact that $I_0^{1-\alpha}u_n(0,x) = 0$.
This proves \eqref{eq0120_03}

Since $\{\bar{u}_n\}$ is Cauchy in $\bH_p^{\alpha,k}\left((S,T) \times \Omega\right)$ and $\bar{u}_n \to \bar{u}$ in $L_p\left((S,T) \times \Omega\right)$, we see that $\bar{u} \in \bH_p^{\alpha,k}\left((S,T) \times \Omega\right)$.
Moreover, since $\bar{u}_n(S,x) = 0$, $\bar{u} \in \bH_{p,0}^{\alpha,k}\left((S,T) \times \Omega\right)$.
In fact, $\bar{u}_n$'s are not necessarily in $C^\infty\left( [S,T] \times \Omega \right)$, but by using mollifications from $\bar{u}_n$ one can easily obtain $v_n \in C^\infty\left( [S,T] \times \Omega \right)$ vanishing for large $|x|$ such that $v_n(S,x) = 0$ and
$$
v_n \to \bar{u} \quad \text{in} \quad \bH_p^{\alpha,k}\left((S,T) \times \Omega\right).
$$
The lemma is proved.
\end{proof}

\begin{lemma}
							\label{lem0207_1}
Let $p \in [1,\infty)$, $\alpha \in (0,1)$, $k \in \{1,2,\ldots\}$, $-\infty < S < t_0 < T < \infty$, and $v \in \bH_p^{\alpha,k}\left((S,T) \times \Omega \right)$.
Then, for any infinitely differentiable function $\eta$ defined on $\bR$ such that $\eta(t)=0$ for $t \leq t_0$ and
$$
|\eta'(t)| \le M, \quad t \in \bR,
$$
the function $\eta v$ belongs to $\bH_{p,0}^{\alpha,k}\left( (t_0,T) \times \Omega \right)$ and
\begin{equation}
                    \label{eq9.45}
\partial_t^\alpha (\eta v)(t,x) = \partial_t I_{t_0}^{1-\alpha} (\eta
v)(t,x) = \eta(t) \partial_t I_S^{1-\alpha} v (t,x) - g(t,x),
\end{equation}
for $(t,x) \in (t_0,T) \times \Omega$,
where
\begin{equation}
							\label{eq0207_04}
g(t,x) = \frac{\alpha}{\Gamma(1-\alpha)} \int_S^t (t-s)^{-\alpha-1} \left(\eta(s) - \eta(t)\right) v(s,x) \, ds
\end{equation}
satisfies
\begin{equation}
							\label{eq0207_01}
\|g\|_{L_p\left((t_0,T) \times \Omega\right)} \le N(\alpha, p, M, T, S) \|v\|_{L_p\left( (S,T) \times \Omega\right)}.
\end{equation}
\end{lemma}

\begin{proof}
As in Lemma \ref{lem0206_1}, we assume that $t_0 = 0$.
First we check \eqref{eq0207_01}.
Note that since $|\eta'(t)| \leq M$, we have
\begin{align*}
&\left| \int_S^t (t-s)^{-\alpha-1} \left(\eta(t) - \eta(s) \right) v(s,x) \, ds \right|\\
&\leq M \int_S^t (t-s)^{-\alpha}|v(s,x)| \, ds = M \Gamma(1-\alpha) I^{1-\alpha}_S |v(t,x)|
\end{align*}
for $(t,x) \in \Omega_T$.
Hence, the inequality \eqref{eq0207_01} follows from  Lemma \ref{lem1018_01} with $1-\alpha$ in place of $\alpha$ (also see Remark \ref{rem0120_1}).

Since $v \in \bH_p^{\alpha,2}\left((S,T) \times \Omega \right)$, there exists a sequence $\{v_n\} \subset \bH_p^{\alpha,2}\left((S,T) \times \Omega \right) \cap C^\infty\left([S,T] \times \Omega \right)$ such that $v_n$ vanishes for large $|x|$ and
$$
\|\partial_t I_S^{1-\alpha} (v_n - v) \|_{L_p\left((S,T) \times \Omega \right)} + \sum_{0 \leq |\beta| \leq 2} \|D_x^\beta(v_n - v)\|_{L_p\left((S,T) \times \Omega\right)} \to 0
$$
as $n \to \infty$.
Let
$$
g_n(t,x)= \eta(t) \partial_t I_S^{1-\alpha} v_n (t,x) - \partial_t I_0^{1-\alpha} (\eta v_n) (t,x).
$$
Then
\begin{align*}
&- \Gamma(1-\alpha) g_n(t,x)\\
&= \partial_t \int_0^t (t-s)^{-\alpha} \eta(s) v_n(s,x) \, ds - \eta(t) \partial_t \int_S^t (t-s)^{-\alpha} v_n(s,x) \, ds\\
&= \frac{\partial}{\partial t}\left[\int_0^t (t-s)^{-\alpha} \eta(s) v_n(s,x) \, ds - \eta(t) \int_S^t (t-s)^{-\alpha} v_n(s,x) \, ds  \right]\\
&\qquad + \eta'(t) \int_S^t (t-s)^{-\alpha} v_n(s,x) \, ds\\
&= \frac{\partial}{\partial t} \left[ \int_S^t (t-s)^{-\alpha} \left(\eta(s) - \eta(t)\right) v_n(s,x) \, ds \right] + \eta'(t) \int_S^t (t-s)^{-\alpha} v_n(s,x) \, ds\\
&= -\alpha \int_S^t (t-s)^{-\alpha-1} \left(\eta(s) - \eta(t)\right) v_n(s,x) \, ds.
\end{align*}
Hence,
$$
g_n(t,x) = \frac{\alpha}{\Gamma(1-\alpha)} \int_S^t (t-s)^{-\alpha-1} \left( \eta(s) - \eta(t) \right) v_n(s,x) \, ds
$$
for $(t,x) \in \Omega_T$.
Clearly,
$$
\eta(t) \partial_t I_S^{1-\alpha} v_n(t,x) \to \eta(t) \partial_t I_S^{1-\alpha} v(t,x)
$$
in $L_p(\Omega_T)$.
From the estimate for $g$ with $v_n - v$ in place of $v$, it follows that
$$
\left\| g_n - g \right\|_{L_p(\Omega_T)} \to 0
$$
as $n \to \infty$.
That is, $$
\partial_t I_0^{1-\alpha}(\eta v_n)(t,x) \to \eta(t) \partial_t I_S^{1-\alpha} v(t,x) - g(t,x)
$$
in $L_p(\Omega_T)$.
This together with $I_0^{1-\alpha} (\eta v_n) \to I_0^{1-\alpha} (\eta v)$ in $L_p(\Omega_T)$
implies \eqref{eq9.45} and $\partial_t I_0^{1-\alpha}(\eta v_n)(t,x) \to \partial_t I_0^{1-\alpha}(\eta v)(t,x)$ in $L_p(\Omega_T)$.
Obviously, $D_x^{\beta} (\eta v_n) \to D_x^\beta (\eta v)$ in $L_p(\Omega_T)$ for $0 \leq |\beta| \leq k$.
Then from the fact that $\eta v_n \in C_0^\infty\left([0,T] \times \Omega \right)$ vanishing for large $|x|$ with $(\eta v_n)(0,x) = 0$, we conclude that $\eta v \in \bH_{p,0}^{\alpha,k}(\Omega_T)$.
\end{proof}

\section{Auxiliary results}
                                    \label{sec4}
Throughout this section, we assume that $a^{ij}$ are measurable functions of only $t \in \bR$. That is,
$a^{ij} = a^{ij}(t)$.

\begin{proposition}
							\label{prop0720_1}
Theorem \ref{thm0412_1} holds when $p=2$.
\end{proposition}

\begin{proof}
A version of this result for divergence type equations can be found in \cite{MR2538276}.
Roughly speaking, the results in this proposition can be obtained by taking the spatial derivatives of the equation in \cite{MR2538276}.
For the reader's convenience, we present here a detailed proof.

By the results from \cite{MR3581300} and the method of continuity, we only prove the a priori estimate \eqref{eq0411_04}.
Moreover, since infinitely differentiable functions with compact support in $x$ and with the zero initial condition are dense in $\bH_{2,0}^{\alpha,2}(\bR^d_T)$, it suffices to prove \eqref{eq0411_04} for $u$ in $C_0^\infty\left([0,T] \times \bR^d\right)$ satisfying $u(0,x) =0$ and \eqref{eq0411_03}.
Multiplying both sides of \eqref{eq0411_03} by $\Delta u$ and then integrating on $(0,T) \times \bR^d$, we have
\begin{equation}
							\label{eq0125_02}
- \int_{\bR^d_T} \partial_t^\alpha u \Delta u \, dx \, dt + \int_{\bR^d_T} a^{ij}(t) D_{ij} u \Delta u \, dx \, dt = \int_{\bR^d_T} f \Delta u \, dx \, dt.
\end{equation}
By integration by parts and the ellipticity condition, it follows that
\begin{align*}
&\int_{\bR^d_T} a^{ij}(t) D_{ij} u \Delta u \, dx \, dt = \int_{\bR^d_T} \sum_{k=1}^d\sum_{i,j=1}^d a^{ij}(t) D_{ij} u D_k^2 u \, dx \, dt\\
&= \int_{\bR^d_T} \sum_{k=1}^d \sum_{i,j=1}^d a^{ij}(t) D_{ki} D_{kj}u \, dx \, dt
\geq \delta \int_{\bR^d_T} \sum_{i,k = 1}^d |D_{ki}u|^2 \, dx \, dt.
\end{align*}
The term on the right-hand side of \eqref{eq0125_02} is taken care of by Young's inequality.
Moreover, the estimate for the term $\partial_t^\alpha u$ follows from that of $D^2u$ and the equation.
Thus, to obtain \eqref{eq0411_04} we only need to see that the first integral in \eqref{eq0125_02} is non-negative.
To do this, by setting $\nabla u = v$, we have
$$
- \int_{\bR^d_T} \partial_t^\alpha u \, \Delta u \, dx \, dt = \int_{\bR^d_T} \partial_t^\alpha v \cdot v \, dx \, dt.
$$
We claim that, for each $(t,x) \in \bR^d_T$,
\begin{equation}
							\label{eq0904_01}
\partial_t^\alpha v(t,x) \cdot v(t,x)
\ge \frac{1}{2} \partial_t^\alpha |v|^2(t,x) .
\end{equation}
To see this, for fixed $t\in (0,T)$ and $x\in \bR^d$, let
$$
F_1(s)=\frac 1 2 |v(s,x)|^2,\quad F_2(s)=v(s,x)\cdot v(t,x),
$$
and
$$
F(s)=\frac 1 2 (|v(s,x)|^2-|v(t,x)|^2)-(v(s,x)-v(t,x))\cdot v(t,x).
$$
Because
$$
F(s)=\frac 1 2|v(s,x)-v(t,x)|^2\ge 0
$$
on $[0,T]$ with the equality at $s=t$, integration by parts clearly yields that
$$
\int_0^t (t-s)^{-\alpha}(F_1'(s)-F_2'(s))\,ds
=\int_0^t (t-s)^{-\alpha}F'(s)\,ds\le 0,
$$
which together with the definition of $\partial_t^\alpha$ implies \eqref{eq0904_01}.
Therefore, because $F_1(0)=0$ we have
\begin{align*}
&2 \Gamma(1-\alpha)\int_0^T\partial_t^\alpha v(t,x) \cdot v(t,x)\,dt\\
&\geq \int_0^T \frac{\partial}{\partial t}\left[ \int_0^t (t-s)^{-\alpha} |v(s,x)|^2 \, ds \right] \, dt =
\left[\int_0^t (t-s)^{-\alpha} |v(s,x)|^2 \, ds\right]_{t=0}^{t=T}\\
&= \int_0^T (T-s)^{-\alpha} |v(s,x)|^2 \, ds \geq 0,
\end{align*}
where we used the fact that $v(s,x)$ is bounded on $[0,T] \times \bR^d$ so that
\begin{align*}
&\int_0^t (t-s)^{-\alpha} |v(s,x)|^2 \, ds
= \int_0^1 (t - tr)^{-\alpha} |v(tr,x)|^2 t \, dr\\
&= t^{1-\alpha} \int_0^1 (1-r)^{-\alpha} |v(tr,x)|^2 \, dr \to 0
\end{align*}
as $t \to 0$.
\end{proof}

\begin{lemma}[Local estimate]
							\label{lem0731_1}
Let $p\in (1,\infty)$, $\alpha \in (0,1)$, $T \in (0,\infty)$, and $0 < r < R < \infty$.
If Theorem \ref{thm0412_1} holds with this $p$ and $v \in \bH_{p,0}^{\alpha,2}\left((0, T) \times B_R\right)$ satisfies
$$
-\partial_t^\alpha v + a^{ij}(t) D_{ij} v
= f
$$
in $(0,T) \times B_R$, then
\begin{align*}
&\| \partial_t^\alpha v \|_{L_p\left((0,T) \times B_r\right)} + \| D^2 v\|_{L_p\left((0,T) \times B_r\right)} \\
&\le \frac{N}{(R-r)^2} \|v\|_{L_p\left((0,T) \times B_R\right)} + N \|f\|_{L_p\left((0,T) \times B_R\right)},
\end{align*}
where $N = N(d,\delta,\alpha,p)$.
\end{lemma}

\begin{proof}
Set
$$
r_0 = r, \quad r_k = r+(R-r)\sum_{j=1}^k \frac{1}{2^j}, \quad k = 1, 2, \ldots.
$$
Let $\zeta_k = \zeta_k(x)$ be an infinitely differentiable function defined on $\bR^d$ such that
$$
\zeta_k = 1 \quad \text{on} \quad B_{r_k}, \quad \zeta = 0 \quad \text{outside} \quad \bR^d \setminus B_{r_{k+1}},
$$
and
$$
|D_x \zeta_k(x)| \le \frac{2^{k+2}}{R-r}, \quad |D^2_x \zeta_k(x)| \le \frac{2^{2k+4}}{(R-r)^2}.
$$
Then $v \zeta_k$ belongs to $\bH_{p,0}^{\alpha,2}(\bR^d_T)$ and satisfies
$$
-\partial_t^\alpha (v \zeta_k) + a^{ij}D_{ij} (v \zeta_k)
= 2 a^{ij} D_i v D_j \zeta_k + a^{ij} v D_{ij} \zeta_k + f\zeta_k
$$
in $(0,T) \times \bR^d$.
By Theorem \ref{thm0412_1}, it follows that
\begin{align}
							\label{eq0728_01}
&\|D^2 (v\zeta_k)\|_{L_p(\bR^d_T)}\nonumber\\
&\le \frac{N2^k}{R-r} \|Dv\|_{L_p\left((0,T)\times B_{r_{k+1}}\right)} + \frac{N2^{2k}}{(R-r)^2} \|v\|_{L_p\left((0,T)\times B_{r_{k+1}}\right)} + N \|f\|_{L_p\left((0,T)\times B_R\right)}\nonumber
\\
&\le \frac{N 2^k}{R-r} \|D(v\zeta_{k+1})\|_{L_p(\bR^d_T)} + \frac{N 2^{2k}}{(R-r)^2} \|v\|_{L_p\left((0,T)\times B_R\right)} + N \|f\|_{L_p\left((0,T)\times B_R\right)},	
\end{align}
where $N = N(d,\delta,\alpha,p)$.
By an interpolation inequality with respect to the spatial variables,
\begin{align*}
&\frac{2^k}{R-r}\|D(v\zeta_{k+1})\|_{L_p(\bR^d_T)}\\
&\le \varepsilon \|D^2(v \zeta_{k+1})\|_{L_p(\bR^d_T)} + N \varepsilon^{-1} \frac{2^{2k}}{(R-r)^2} \|v \zeta_{k+1}\|_{L_p(\bR^d_T)}
\end{align*}
for any $\varepsilon \in (0,1)$, where $N=N(d,p)$.
Combining this inequality with \eqref{eq0728_01}, we obtain that
\begin{align*}
&\|D^2 (v \zeta_k) \|_{L_p(\bR^d_T)}\\
&\le \varepsilon \|D^2 (v \zeta_{k+1}) \|_{L_p(\bR^d_T)} + N \varepsilon^{-1} \frac{4^k}{(R-r)^2}\|v\|_{L_p\left((0,T) \times B_R\right)} + N \|f\|_{L_p\left((0,T)\times B_R\right)},
\end{align*}
where $N = N(d,\delta,\alpha,p)$.
By multiplying both sides of the above inequality by $\varepsilon^k$ and making summation with respect to $k = 0, 1, \ldots$, we see that
\begin{align*}
&\|D^2(v\zeta_0)\|_{L_p(\bR^d_T)} + \sum_{k=1}^\infty \varepsilon^k \|D^2(v\zeta_k)\|_{L_p(\bR^d_T)}\\
&\le \sum_{k=1}^\infty \varepsilon^k \|D^2(v\zeta_k)\|_{L_p(\bR^d_T)} + N \frac{\varepsilon^{-1}}{(R-r)^2} \|v\|_{L_p\left((0,T) \times B_R\right)} \sum_{k=0}^\infty (4\varepsilon)^k\\
&+ N \|f\|_{L_p\left((0,T)\times B_R\right)} \sum_{k=0}^\infty \varepsilon^k,
\end{align*}
where the convergence of the summations are guaranteed if $\varepsilon = 1/8$.
We then obtain the desired inequality in the lemma after we remove the same terms from both sides of the above inequality and use the fact that $\zeta_0 = 1$ on $B_r$.
\end{proof}

\begin{lemma}
							\label{lem0211_1}
Let $p \in [1,\infty)$, $\alpha \in (0,1)$, $0 < T < \infty$, and $0 < r < R < \infty$.
If $v \in \bH_{p,0}^{\alpha,2}\left((0, T) \times B_R\right)$, then, for $\varepsilon \in (0, R-r)$,
$$
D_x v^{(\varepsilon)}(t,x) \in \bH_{p,0}^{\alpha,2}\left((0,T) \times B_r\right),
$$
where $v^{(\varepsilon)}$ is a mollification of $v$ with respect to the spatial variables, that is,
$$
v^{(\varepsilon)}(t,x) = \int_{B_R} \phi_\varepsilon(x-y) v(t,y) \, dy, \quad \phi_\varepsilon(x) = \varepsilon^{-d} \phi(x/\varepsilon),
$$
and $\phi\in C_0^\infty(B_1)$ is a smooth function with unit integral.
\end{lemma}

\begin{proof}
Since $v \in \bH_{p,0}^{\alpha,2}\left((0,T) \times B_R\right)$, there exists a sequence $\{v_n\} \subset C^\infty\big([0,T] \times B_R\big)$ such that $v_n(0,x) = 0$ and
$$
\left\| v_n - v \right\|_{\bH_p^{\alpha,2}\left((0,T) \times B_R \right)} \to 0
$$
as $n \to \infty$.
Then, $D_x v_n^{(\varepsilon)} \in C^\infty\left([0,T]\times B_r\right)$ and $D_x v_n^{(\varepsilon)}(0,x) = 0$.
For $(t,x) \in (0,T) \times B_r$, we have
$$
D_x^k D_x v^{(\varepsilon)}(t,x) = \int_{B_R} (D_x \phi_\varepsilon)(x-y) D_x^k v(t,y) \, dy, \quad k = 0,1,2,
$$
$$
D_t^\alpha D_x v^{(\varepsilon)}(t,x) = \int_{B_R} (D_x \phi_\varepsilon)(x-y) \partial_t^\alpha v(t,y) \, dy.
$$
We also have the same expressions for $v_n$ in place of $v$.
Hence, we see that
$$
\left\| D_x v_n^{(\varepsilon)} - D_x v^{(\varepsilon)} \right\|_{\bH_p^{\alpha,2}\left((0,T) \times B_r \right)} \to 0
$$
as $n \to \infty$.
This shows that $D_x v^{(\varepsilon)} \in \bH_{p,0}^{\alpha,2}\left((0,T) \times B_r\right)$.
\end{proof}

If $v \in \bH_{p,0}^{\alpha,2}\left((S,T) \times \bR^d\right)$ is a solution to a homogenous equation, one can improve its regularity as follows.

\begin{lemma}
							\label{lem0731_2}
Let $p\in (1,\infty)$, $\alpha \in (0,1)$, $-\infty< S < t_0 < T < \infty$, and $0 < r < R < \infty$.
Suppose that Theorem \ref{thm0412_1} holds with this $p$ and $v \in \bH_{p,0}^{\alpha,2}\left((S, T) \times B_R\right)$ satisfies
$$
-\partial_t^\alpha v + a^{ij}(t) D_{ij} v
= f
$$
in $(S,T) \times B_R$, where $f(t,x) = 0$ on $(t_0,T) \times B_R$ and, as we recall,
$$
\partial_t^\alpha v (t,x) = \frac{\partial}{\partial t} I_S^{1-\alpha} v(t,x) = \frac{1}{\Gamma(1-\alpha)} \frac{\partial}{\partial t} \int_S^t (t-s)^{-\alpha} v(s,x) \, ds.
$$
Then, for any infinitely differentiable function $\eta$ defined on $\bR$ such that $\eta(t)=0$ for $t \leq t_0$, the function
$D^2 (\eta v) = D^2_x (\eta v)$ belongs to $\bH_{p,0}^{\alpha,2}\left( (t_0,T) \times B_r \right)$ and satisfies
$$
- \partial_t^\alpha (D^2(\eta v)) + a^{ij}(t) D_{ij} (D^2(\eta v)) = \cG
$$
in $(t_0,T) \times B_r$,
where $\partial_t^\alpha = \partial_t I_{t_0}^{1-\alpha}$ and $\cG$ is defined by
$$
\cG(t,x) = \frac{\alpha}{\Gamma(1-\alpha)} \int_S^t (t-s)^{-\alpha-1}\left(\eta(t) - \eta(s)\right) D^2 v(s,x) \, ds.
$$
Moreover,
\begin{equation}
							\label{eq0208_02}
\|D^4 (\eta v)\|_{L_p\left((0,T) \times B_r\right)} \le \frac{N}{(R-r)^2} \|D^2 v\|_{L_p\left((0,T) \times B_R\right)} + N \|\cG\|_{L_p\left((0,T) \times B_R\right)},
\end{equation}
where $N = N(d,\delta,\alpha,p)$.
\end{lemma}

\begin{proof}
Without loss of generality we assume   $t_0 = 0$ so that
$$
- \infty < S < 0 < T < \infty.
$$
By Lemma \ref{lem0207_1} and the fact that $f(t,x) = 0$ on $(0,T) \times B_R$, we have that $\eta v$ belongs to $\bH_{p,0}^{\alpha,2}\left((0,T) \times B_R\right)$ and satisfies
$$
- \partial_t^\alpha (\eta v) + a^{ij}(t) D_{ij}(\eta v) = g
$$
in $(0,T) \times B_R$, where $g \in L_p\left((0,T) \times B_R\right)$ is from \eqref{eq0207_04}.

Find $r_i$, $i=1,2,3$, such that
$r = r_1 < r_2 < r_3 < R$.
Set $w = \eta v$ and
consider
$w^{(\varepsilon)}$, $\varepsilon \in (0,R-r_3)$, from Lemma \ref{lem0211_1}, which is a mollification of $w$ with respect to the spatial variables.
Since $w \in \bH_{p,0}^{\alpha,2}\left((0,T) \times B_R\right)$, by Lemma \ref{lem0211_1}, $D_x w^{(\varepsilon)}$ belongs to $\bH_{p,0}^{\alpha,2}\left((0,T) \times B_{r_3}\right)$ and satisfies
$$
- \partial_t^\alpha (D_x w^{(\varepsilon)}
) + a^{ij}(t) D_{ij} (D_x w^{(\varepsilon)} w) = D_x g^{(\varepsilon)}
$$
in $(0,T) \times B_{r_3}$, where
$$
D_x g^{(\varepsilon)}(t,x) = \frac{\alpha}{\Gamma(1-\alpha)} \int_S^t (t-s)^{-\alpha-1} \left(\eta(s) - \eta(t)\right) D_x v^{(\varepsilon)}(s,x) \, ds.
$$
It then follows from Lemma \ref{lem0731_1} that
\begin{multline}
							\label{eq0209_01}
\| \partial_t^\alpha (D_x w^{(\varepsilon)})\|_{L_p\left((0,T) \times B_{r_2}\right)} + \| D^2 (D_x w^{(\varepsilon)})\|_{L_p\left((0,T) \times B_{r_2}\right)}
\\
\le \frac{N}{(r_3-r_2)^2} \|D_x w^{(\varepsilon)}\|_{L_p\left((0,T) \times B_{r_3}\right)} + N \|D_x g^{(\varepsilon)}\|_{L_p\left((0,T) \times B_{r_3}\right)},
\end{multline}
where $N = N(d,\delta,p)$.
Note that
\begin{equation}
							\label{eq0211_01}
\|D_x w^{(\varepsilon)} - D_x w\|_{L_p\left((0,T) \times B_{r_3}\right)} \to 0, \quad \|D_x g^{(\varepsilon)} - \cG_0 \|_{L_p\left((0,T) \times B_{r_3}\right)} \to 0,
\end{equation}
where $\cG_0$ is defined as $\cG$ with $Dv$ in place of $D^2 v$. In particular, the latter convergence in \eqref{eq0211_01} is guaranteed by \eqref{eq0207_01} and the properties of mollifications.
Recall that $D_x w^{(\varepsilon)} \in \bH_{p,0}^{\alpha,2}\left((0,T) \times B_{r_3}\right)$.
Then, from \eqref{eq0209_01} and \eqref{eq0211_01}, we conclude that $D_x w$ belongs to $\bH_{p,0}^{\alpha,2}\left((0,T) \times B_{r_2}\right)$ and satisfies
$$
- \partial_t^\alpha (D_x w ) + a^{ij}(t) D_{ij} (D_x w) = \cG_0
$$
in $(0,T) \times B_{r_2}$.
We now repeat the above argument with $Dw$, $r_1$, and $r_2$ in place of $w$, $r_2$, and $r_3$, respectively, along with the observation that the limits in \eqref{eq0211_01} hold with $Dw$ in place of $w$.
In particular, the estimate \eqref{eq0209_01} with $Dw$ in place of $w$ implies \eqref{eq0208_02}.
The lemma is proved.
\end{proof}

\section{Level set arguments}
                        \label{sec5}

Recall that $Q_{R_1,R_2}(t,x) = (t-R_1^{2/\alpha}, t) \times B_{R_2}(x)$ and  $Q_R(t,x)=Q_{R,R}(t,x)$.
For $(t_0,x_0) \in \bR \times \bR^d$ and a function $g$ defined on $(-\infty,T) \times \bR^d$, we set
\begin{equation}
							\label{eq0406_03b}
\cM g(t_0,x_0) = \sup_{Q_{R}(t,x) \ni (t_0,x_0)} \dashint_{Q_{R}(t,x)}|g(s,y)| I_{(-\infty,T) \times \bR^d} \, dy \, ds
\end{equation}
and
\begin{equation}
							\label{eq0406_03}
\cS\cM g(t_0,x_0) = \sup_{Q_{R_1,R_2}(t,x) \ni (t_0,x_0)} \dashint_{Q_{R_1,R_2}(t,x)}|g(s,y)| I_{(-\infty,T) \times \bR^d} \, dy \, ds.
\end{equation}
The first one is called the (parabolic) maximal function of $g$, and second one is called the strong (parabolic) maximal function of $g$.

\begin{proposition}
							\label{prop0406_1}
Let $p\in (1,\infty)$, $\alpha \in (0,1)$, $T \in (0,\infty)$, and $a^{ij} = a^{ij}(t)$.
Assume that Theorem \ref{thm0412_1} holds with this $p$ and $u \in \bH_{p,0}^{\alpha,2}(\bR^d_T)$ satisfies
$$
-\partial_t^\alpha u + a^{ij} D_{ij} u
= f
$$
in $(0,T) \times \bR^d$.
Then there exists
$$
p_1 = p_1(d, \alpha,p)\in (p,\infty]
$$
satisfying
\begin{equation}
							\label{eq0411_05}
p_1 > p + \min\left\{\frac{2\alpha}{\alpha d + 2 - 2\alpha}, \alpha, \frac{2}{d} \right\}
\end{equation}
and the following.
For $(t_0,x_0) \in [0,T] \times \bR^d$ and $R \in (0,\infty)$,
there exist
$$
w \in \bH_{p,0}^{\alpha,2}((t_0-R^{2/\alpha}, t_0)\times \bR^d), \quad v \in \bH_{p,0}^{\alpha,2}((S,t_0) \times \bR^d),
$$
where $S = \min\{0, t_0 - R^{2/\alpha}\}$,
such that $u = w + v$ in $Q_R(t_0,x_0)$,
\begin{equation}
                                    \label{eq8.13}
\left( |D^2w|^p \right)_{Q_R(t_0,x_0)}^{1/p} \le N \left( |f|^p \right)_{Q_{2R}(t_0,x_0)}^{1/p},
\end{equation}
and
\begin{multline}
							\label{eq0411_01}
\left( |D^2v|^{p_1} \right)_{Q_{R/2}(t_0,x_0)}^{1/p_1} \leq N \left( |f|^p \right)_{Q_{2R}(t_0,x_0)}^{1/p}
\\
+ N \sum_{k=0}^\infty 2^{-k\alpha} \left( \dashint_{\!t_0 - (2^{k+1}+1)R^{2/\alpha}}^{\,\,\,t_0} \dashint_{B_R(x_0)} |D^2u(s,y)|^p \, dy \, ds \right)^{1/p},
\end{multline}
where $N=N(d,\delta, \alpha,p)$.
Here we understand that $u$ and $f$ are extended to be zero whenever $t < 0$
and
$$
\left( |D^2v|^{p_1} \right)_{Q_{R/2}(t_0,x_0)}^{1/p_1} = \|D^2v\|_{L_\infty(Q_{R/2}(t_0,x_0))},
$$
provided that $p_1 = \infty$.
\end{proposition}

\begin{proof}
We extend $u$ and $f$ to be zero, again denoted by $u$ and $f$, on $(-\infty,0) \times \bR^d$.
Thanks to translation, it suffices to prove the desired inequalities when $x_0 = 0$.
Moreover, we assume that $R = 1$.
Indeed, for $R > 0$, we set
$$
\tilde{u}(t,x) = R^{-2}u(R^{2/\alpha}t, R x), \quad \tilde{a}^{ij} = a^{ij}(R^{2/\alpha}t), \quad \tilde{f}(t,x) = f(R^{2/\alpha}t, Rx).
$$
Then
$$
- \partial_t^\alpha \tilde{u} + \tilde{a}^{ij}(t) D_{ij} \tilde{u} = \tilde{f}
$$
in $(0,R^{-2/\alpha} T) \times \bR^d$.
We then apply the result for $R=1$ to this equation on
$$
(\tilde{t}_0-1,\tilde{t}_0) \times B_1, \quad \tilde{t}_0 = R^{-2/\alpha} t_0
$$
and return to $u$.

For $R=1$ and $t_0 \in (0,\infty)$, set $\zeta = \zeta(t,x)$ to be an infinitely differentiable function defined on $\bR^{d+1}$ such that
$$
\zeta = 1 \quad \text{on} \quad (t_0-1, t_0) \times B_1,
$$
and
$$
\zeta = 0 \quad \text{on} \quad \bR^{d+1} \setminus (t_0-2^{2/\alpha}, t_0+2^{2/\alpha}) \times B_2.
$$
Using Theorem \ref{thm0412_1}, find a solution $w\in \bH_{p,0}^{\alpha,2}(\bR^d_T)$ to the problem
$$
\left\{
\begin{aligned}
-\partial_t^\alpha w + a^{ij}(t) D_{ij} w &= \zeta(t,x) f(t,x)\quad  \text{in} \,\, (t_0-1, t_0) \times \bR^d,
\\
w(t_0 - 1,x) &= 0 \quad  \text{on} \quad \bR^d.
\end{aligned}
\right.
$$
where we recall that
$$
\partial_t^\alpha w = \frac{1}{\Gamma(1-\alpha)} \partial_t \int_{t_0-1}^t (t-s)^{-\alpha} w(s,x) \, ds.
$$
Again extend $w$ to be zero on $(-\infty,t_0-1) \times \bR^d$.
From Theorem \ref{thm0412_1} it follows that
\begin{equation}
							\label{eq1214_01}
\|\partial_t^\alpha w\|_{L_p\left(Q_r(t_0,0)\right)} + \|D^2 w \|_{L_p\left(Q_r(t_0,0)\right)}
\le N \|f\|_{L_p\left(Q_2(t_0,0)\right)}
\end{equation}
for any $r > 0$.

Set $v = u - w$ so that
$$
v =
\left\{
\begin{aligned}
u-w, &\quad t \in (t_0 - 1,t_0),
\\
u, &\quad t \in (-\infty, t_0 - 1],
\end{aligned}
\right.
$$
where we note that it is possible to have $t_0 - 1 < 0$.
Then by Lemma \ref{lem0206_1}, $v$ belongs to $\bH_{p,0}^{\alpha,2}\left((S,t_0) \times \bR^d\right)$ for $S := \min \{0, t_0 -1\}$ and satisfies
$$
\partial_t^\alpha w = \partial_t I_{t_0-1}^{1-\alpha} w = \partial_t I_S^{1-\alpha} w, \quad \partial_t^\alpha u = \partial_t I_0^{1-\alpha} u = \partial_t I_S^{1-\alpha} u,
$$
and
$$
-\partial_t^\alpha v + a^{ij} D_{ij} v = h
$$
in $(S,t_0)\times \bR^d$, where
$$
h(t,x) = \left\{
\begin{aligned}
\left( 1 -\zeta(t,x)\right) f(t,x) \quad &\text{in} \,\, (t_0 - 1, t_0) \times \bR^d,
\\
f(t,x) \quad &\text{in} \,\, (S, t_0 - 1) \times \bR^d.
\end{aligned}
\right.
$$
In particular, we note that $h = 0$
in $(t_0 - 1,t_0) \times B_1$.

Find an infinitely differentiable  function $\eta$ defined on $\bR$ such that
$$
\eta =
\left\{
\begin{aligned}
1 \quad &\text{if} \quad t \in (t_0-(1/2)^{2/\alpha},t_0),
\\
0 \quad &\text{if} \quad t \in \bR \setminus (t_0-1,t_0+1),
\end{aligned}
\right.
$$
and
$$
\left|\frac{\eta(t)-\eta(s)}{t-s}\right| \le N(\alpha).
$$
By Lemma \ref{lem0731_2}, $D^2(\eta v)$ belongs to $\bH_{p,0}^{\alpha,2}\left((t_0-1,t_0)\times B_{3/4}\right)$ and satisfies
$$
-\partial_t^\alpha \left( D^2(\eta v) \right) + a^{ij} D_{ij} D^2(\eta v) = \cG
$$
in $(t_0-1,t_0)\times B_{3/4}$,
where
$$
\cG(t,x) = \frac{\alpha}{\Gamma(1-\alpha)} \int_S^t (t-s)^{-\alpha-1}\left(\eta(t) - \eta(s)\right) D^2 v(s,x) \, ds.
$$

If $p \leq 1/\alpha$, take $p_1$ satisfying
$$
p_1 \in \left(p, \frac{1/\alpha + d/2}{1/(\alpha p) + d/(2 p) -1}\right) \quad \text{if} \quad p \leq d/2,
$$
$$
p_1 \in \left(p, p(\alpha p + 1)\right) \quad \text{if} \quad p > d/2.
$$
If $p > 1/\alpha$, take $p_1$ satisfying
$$
p_1 \in \left(p, p + 2p^2/d\right) \quad \text{if} \quad p \leq d/2,
$$
$$
p_1 \in (p, 2p) \quad \text{if} \quad p > d/2, \quad p \leq d/2 + 1/\alpha,
$$
$$
p_1 = \infty \quad \text{if} \quad p > d/2 + 1/\alpha.
$$
Note that $p_1$ satisfies \eqref{eq0411_05}
and the increment $\min \{2\alpha/(\alpha d + 2 - 2\alpha), \alpha, 2/d\}$ is independent of $p$.
By Lemma \ref{lem0731_2} and the embedding results in Appendix (Corollary \ref{cor1211_1}, Theorem \ref{thm1207_2}, Corollary \ref{cor0225_1}, Theorem \ref{thm0214_1}, and Theorem \ref{thm5.18}), we have
\begin{align}
							\label{eq0715_01}
&\|D^2 v\|_{L_{p_1}\left(Q_{1/2}(t_0,0)\right)} \le \|D^2(\eta v)\|_{L_{p_1}\left((t_0-1,t_0)\times B_{1/2}\right)}\nonumber
\\
&\le N\| |D^2(\eta v)| + |D^4(\eta v)| + |D_t^\alpha D^2 (\eta v)| \|_{L_p\left((t_0-1,t_0)\times B_{3/4}\right)}\nonumber
\\
&\le N \||D^2(\eta v)| + |\cG|\|_{L_p\left((t_0-1,t_0)\times B_1\right)} \le N \||D^2 v| + |\cG|\|_{L_p\left((t_0-1,t_0)\times B_1\right)}\nonumber
\\
&\leq N \| |D^2 u| + |D^2 w| + |\cG| \|_{L_p\left((t_0-1,t_0)\times B_1\right)},
\end{align}
where $N = N(d, \delta,
\alpha,p,p_1)$ and we used the fact that
$$
D_t^\alpha D^2(\eta v) = a^{ij} D_{ij} D^2 (\eta v) - \cG
$$
in $(t_0-1,t_0) \times B_{3/4}$.

Since $D^2 v = 0$ for $t \leq S$, we write
\begin{align*}
&\frac{\Gamma(1-\alpha)}{\alpha} \cG(t,x) = \int_{-\infty}^t (t-s)^{-\alpha-1} \left( \eta(s) - \eta(t) \right) D^2 v(s,x) \, ds\\
&= \int_{t-1}^t (t-s)^{-\alpha-1}\left( \eta(s) - \eta(t) \right) D^2 v(s,x) \, ds\\
&\quad + \int_{-\infty}^{t-1} (t-s)^{-\alpha-1}\left( \eta(s) - \eta(t) \right) D^2 v(s,x) \, ds := I_1(t,x)+I_2(t,x),
\end{align*}
where
\begin{align*}
|I_1(t,x)| \le N \int_{t-1}^t |t-s|^{-\alpha} |D^2 v(s,x)|\,ds= N \int_0^1 |s|^{-\alpha} |D^2 v(t-s,x)|\, ds,
\end{align*}
From this we have
\begin{multline}
							\label{eq0715_02}
\|I_1\|_{L_p\left((t_0-1,t_0) \times B_1\right)} \le N \|D^2 v\|_{L_p\left((t_0-2,t_0)\times B_1\right)}
\\
= N \|D^2 v\|_{L_p\left((t_0-1,t_0)\times B_1\right)} + N \|D^2 u\|_{L_p\left((t_0-2,t_0-1)\times B_1\right)}.
\end{multline}
To estimate $I_2$, we see that
$\eta(s) = 0$ for any $s \in (-\infty, t-1)$ with $t \in (t_0-1,t_0)$.
Thus we have
$$
I_2(t,x) = -\eta(t) \int_{-\infty}^{t-1} (t-s)^{-\alpha-1} D^2 v(s,x) \, ds.
$$
Then,
\begin{align*}
|I_2(t,x)| &\le \int_{-\infty}^{t-1} |t-s|^{-\alpha-1} |D^2 v(s,x)| \, ds\\
&= \sum_{k=0}^\infty \int_{t-2^{k+1}}^{t-2^k} |t-s|^{-\alpha-1} |D^2 v(s,x)|\,ds\\
&\le \sum_{k=0}^\infty \int_{t-2^{k+1}}^{t-2^k } 2^{-k(\alpha+1)} |D^2 v(s,x)|\,ds.
\end{align*}
From this we have
\begin{equation*}
\|I_2\|_{L_p\left((t_0-1,t_0) \times B_1\right)}\le \sum_{k=0}^\infty 2^{-k(\alpha+1)} \left\| \int_{t-2^{k+1}}^{t-2^k} |D^2 v(s,x)| \, ds \right\|_{L_p\left((t_0-1,t_0) \times B_1\right)}.
\end{equation*}
Since $t_0 - 1 < t < t_0$,
$$
\int_{t-2^{k+1}}^{t-2^k} |D^2 v(s,x)|\, ds \leq \int_{t_0-(2^{k+1}+1)}^{t_0-2^k} |D^2 v(s,x)|\, ds.
$$
Hence, by the Minkowski inequality,
\begin{align*}
&\left\| \int_{t-2^{k+1}}^{t-2^k } |D^2 v(s,x)| \, ds \right\|_{L_p\left(Q_1(t_0,0)\right)}\\
&\le \int_{t_0 - (2^{k+1}+1)}^{t_0 - 2^k} \left( \int_{B_1} |D^2 v(s,x)|^p \, dx \right)^{1/p} \, ds\\
&\le 2^{k+2}\left( \dashint_{\!t_0 - (2^{k+1}+1)}^{\,\,\,t_0} \dashint_{B_1} |D^2 v(s,x)|^p \, dx \, ds \right)^{1/p}.
\end{align*}
It then follows that
\begin{align*}
&\|I_2\|_{L_p\left(Q_1(t_0,0)\right)}\\
&\le \sum_{k=0}^\infty 2^{-k\alpha+2}\left( \dashint_{\!t_0 - (2^{k+1}+1)}^{\,\,\,t_0} \dashint_{B_1} |D^2 v(s,x)|^p \, dx \, ds \right)^{1/p}\\
&\le \sum_{k=0}^\infty 2^{-k\alpha+2} \left( \dashint_{\!t_0 - (2^{k+1}+1)}^{\,\,\,t_0} \dashint_{B_1} |D^2 u(s,x)|^p \, dx \, ds \right)^{1/p}\\
&\quad + \sum_{k=0}^\infty 2^{-k\alpha+2}\left( \dashint_{\!t_0 - (2^{k+1}+1)}^{\,\,\,t_0} \dashint_{B_1} |D^2 w(s,x)|^p \, dx \, ds \right)^{1/p},
\end{align*}
where
$$
\sum_{k=0}^\infty 2^{-k\alpha+2}\left( \dashint_{\!t_0 - (2^{k+1}+1)}^{\,\,\,t_0} \dashint_{B_1} |D^2 w(s,x)|^p \, dx \, ds \right)^{1/p} \leq N(\alpha) \left( |D^2 w|^p \right)^{1/p}_{Q_1(t_0,0)}.
$$
Combining the above inequalities, \eqref{eq0715_01}, and \eqref{eq0715_02}, we get
\begin{align*}
&\|D^2 v\|_{L_{p_1}\left(Q_{1/2}(t_0,0)\right)} \leq N \left( |D^2 w|^p \right)_{Q_1(t_0,0)}^{1/p}\\
&\quad + N \sum_{k=0}^\infty 2^{-k\alpha} \left( \dashint_{\!t_0 - (2^{k+1}+1)}^{\,\,\,t_0} \dashint_{B_1(x_0)} |D^2u(s,y)|^p \, dy \, ds \right)^{1/p}.
\end{align*}
We then use \eqref{eq1214_01} with $r=1$ to obtain \eqref{eq0411_01} with $R = 1$.
The proposition is proved.
\end{proof}

Let $\gamma\in (0,1)$, and let $p\in (1,\infty)$ and $p_1=p_1(d,\alpha,p)$ be from the above proposition.
Denote
\begin{equation}
							\label{eq0406_04}
\cA(s) = \left\{ (t,x) \in (-\infty,T) \times \bR^d: |D^2 u(t,x)| > s \right\}
\end{equation}
and
\begin{multline}
							\label{eq0406_05}
\cB(s) = \big\{ (t,x) \in (-\infty,T) \times \bR^d:
\\
\gamma^{-1/p}\left( \cM |f|^p (t,x) \right)^{1/p} + \gamma^{-1/p_1}\left( \cS\cM |D^2 u|^p(t,x)\right)^{1/p} > s  \big\},
\end{multline}
where, to well define $\cM$ and $\cS\cM$ (recall the definitions in \eqref{eq0406_03b} and \eqref{eq0406_03}), we extend a given function to be zero for $t \leq S$ if the function is defined on $(S,T) \times \bR^d$.

Set
\begin{equation}
							\label{eq0606_01}
\cC_R(t,x) = (t-R^{2/\alpha},t+R^{2/\alpha}) \times B_R(x),\quad
\hat \cC_R(t,x)=\cC_R(t,x)\cap \{t\le T\}.
\end{equation}

\begin{lemma}
							\label{lem0409_1}
Let $p\in (1,\infty)$, $\alpha \in (0,1)$, $T \in (0,\infty)$, $a^{ij} = a^{ij}(t)$, $R \in (0,\infty)$, and $\gamma \in (0,1)$.
Assume that Theorem \ref{thm0412_1} holds with this $p$ and $u \in \bH_{p,0}^{\alpha,2}(\bR^d_T)$ satisfies
$$
-\partial_t^\alpha u + a^{ij} D_{ij} u
= f
$$
in $(0,T) \times \bR^d$.
Then, there exists a constant $\kappa = \kappa(d,\delta,\alpha,p) > 1$ such that the following hold:
for $(t_0,x_0) \in (-\infty,T] \times \bR^d$ and $s>0$, if
\begin{equation}
							\label{eq0406_02}
|\cC_{R/4}(t_0,x_0) \cap \cA(\kappa s)|  \geq \gamma |\cC_{R/4}(t_0,x_0)|,
\end{equation}
then we have
$$
\hat\cC_{R/4}(t_0,x_0) \subset \cB(s).
$$
\end{lemma}

\begin{proof}
By dividing the equation by $s$, we may assume that $s = 1$.
We only consider $(t_0,x_0) \in (-\infty,T] \times \bR^d$ such that $t_0 + (R/4)^{2/\alpha} \geq 0$, because otherwise,
$$
\cC_{R/4}(t_0,x_0) \cap \cA(\kappa) \subset \left\{ (t,x) \in (-\infty,0] \times \bR^d: |D^2 u(t,x)| > s \right\} = \emptyset
$$
as $u(t,x)$ is extended to be zero for $t < 0$.
Suppose that there is a point $(s,y) \in \hat\cC_{R/4}(t_0,x_0)$
such that
\begin{equation}
							\label{eq0406_01}
\gamma^{-1/p}\left( \cM |f|^p (s,y) \right)^{1/p} + \gamma^{-1/p_1} \left( \cS\cM |D^2 u|^p(s,y)\right)^{1/p} \leq 1.
\end{equation}
Set
$$
t_1 := \min \{ t_0 + (R/4)^{2/\alpha}, T\} \quad \text{and} \quad x_1 := x_0.
$$
Then $(t_1,x_1) \in [0,T] \times \bR^d$ and by Proposition \ref{prop0406_1} there exist $p_1 = p_1(d,\alpha,p) \in (p,\infty]$ and $w \in \cH_{p,0}^{\alpha,2}\left((t_1-R^{2/\alpha}, t_1) \times \bR^d\right)$, $v \in \cH_{p,0}^{\alpha,2}\left((S, t_1) \times \bR^d\right)$, where $S=\min\{0,t_1-R^{2/\alpha}\}$, such that $u = w + v$ in $Q_R(t_1,x_1)$,
\begin{equation}
							\label{eq0409_01}
\left(|D^2 w|^p\right)^{1/p}_{Q_R(t_1,x_1)} \leq N \left( |f|^p \right)_{Q_{2R}(t_1,x_1)}^{1/p},
\end{equation}
and
\begin{multline}
							\label{eq0409_02}
\left( |D^2v|^{p_1} \right)_{Q_{R/2}(t_1,x_1)}^{1/p_1} \le  N \left( |f|^p \right)_{Q_{2R}(t_1,x_1)}^{1/p}
\\
+ N \sum_{k=0}^\infty 2^{-\kappa \alpha} \left(\dashint_{\! t_1 - (2^{k+1}+1)R^{2/\alpha}}^{\,\,\,t_1} \dashint_{B_R(x_1)} |D^2u(\ell,z)|^p \, dz \, d \ell \right)^{1/p},
\end{multline}
where $N=N(d,\delta, \alpha,p)$.
Since $t_0 \leq T$, we have
$$
(s,y) \in \hat\cC_{R/4}(t_0,x_0)
\subset Q_{R/2}(t_1,x_1) \subset Q_{2R}(t_1,x_1),
$$
$$
(s,y) \in \hat\cC_{R/4}(t_0,x_0) \subset (t_1- (2^{k+1}+1)R^{2/\alpha}, t_1) \times B_R(x_1)
$$
for all $k = 0,1,\ldots$.
From these set inclusions, in particular, we observe that
$$
\dashint_{\! t_1 - (2^{k+1}+1)R^{2/\alpha}}^{\,\,\,t_1} \dashint_{B_R(x_1)} |D^2u(\ell,z)|^p \, dz \, d \ell \leq \cS\cM |D^2u|^p(s,y)
$$
for all $k=0,1,2,\ldots$.
Thus the inequality \eqref{eq0406_01} along with \eqref{eq0409_01} and \eqref{eq0409_02} implies that
$$
\left( |D^2v|^{p_1} \right)_{Q_{R/2}(t_1,x_1)}^{1/p_1} \leq N \gamma^{1/p} + N \gamma^{1/p_1} \leq N_0 \gamma^{1/p_1},
$$
$$
\left(|D^2 w|^p\right)^{1/p}_{Q_R(t_1,x_1)} \leq N_1 \gamma^{1/p},
$$
where $N_0$ and $N_1$ depend only on $d$, $\delta$, $\alpha$, and $p$.
Note that, for a sufficiently large $K_1$,
\begin{align*}
&|\cC_{R/4}(t_0,x_0) \cap \cA(\kappa)|
= |\{(t,x) \in \cC_{R/4}(t_0,x_0), t \in (-\infty,T): |D^2u(t,x)| > \kappa\}|\\
&\leq \left|\{ (t,x) \in Q_{R/2}(t_1,x_1): |D^2 u(t,x)| > \kappa\}\right|\\
&\leq \left|\{(t,x) \in Q_{R/2}(t_1,x_1): |D^2 w(t,x)| > \kappa - K_1 \}\right|\\
&\quad + \left|\{(t,x) \in Q_{R/2}(t_1,x_1): |D^2 v(t,x)| > K_1 \}\right|\\
&\leq (\kappa-K_1)^{-p} \int_{Q_{R/2}(t_1,x_1)} |D^2 w|^p \, dx \, dt + K_1^{-p_1} \int_{Q_{R/2}(t_1,x_1)} |D^2v|^{p_1} \, dx \, dt\\
&\leq \frac{N_1^p \gamma |Q_R|}{(\kappa - K_1)^p} + \frac{N_0^{p_1}\gamma|Q_{R/2}|}{K_1^{p_1}}I_{p_1 \neq \infty}\\
&\leq N(d,\alpha) |\cC_{R/4}| \left(\frac{N_1^p \gamma }{(\kappa - K_1)^p} + \gamma \left(\frac{N_0}{K_1}\right)^p I_{p_1 \neq \infty} \right) < \gamma |\cC_{R/4}(t_0,x_0)|,
\end{align*}
provided that we choose a sufficiently large $K_1(\ge N_0)$ depending only on $d$, $\delta$, $\alpha$, and $p$, so that
$$
N(d,\alpha)  (N_0/K_1)^p < 1/2,
$$
and then choose a $\kappa$ depending only on $d$, $\delta$, $\alpha$, and $p$, so that
$$
N(d,\alpha) N_1^p/(\kappa-K_1)^p < 1/2.
$$
Considering \eqref{eq0406_02}, we get a contradiction.
The lemma is proved.
\end{proof}

\section{$L_p$-estimates}
                    \label{sec6}

Now we are ready to give the proof of Theorem \ref{thm0412_1}.

\begin{proof}[Proof of Theorem \ref{thm0412_1}]
We first consider the case when $p\in [2,\infty)$ by using an iterative argument to successively increase the exponent $p$. When $p=2$, the theorem follows from Proposition \ref{prop0720_1}.
Now suppose that the theorem is proved for some $p_0\in [2,\infty)$. Let $p_1=p_1(d,\alpha,p_0)$ be from Proposition \ref{prop0406_1}, and $p\in (p_0,p_1)$.
As in the proof of Proposition \ref{prop0720_1} we assume $u \in C_0^\infty\left([0,T] \times \bR^d\right)$ with $u(0,x) = 0$ and prove the a priori estimate \eqref{eq0411_04}.
Note that
\begin{equation}
                                    \label{eq8.47}
\|D^2u\|_{L_p(\bR^d_T)}^p = p \int_0^\infty |\cA(s)| s^{p-1} \, ds = p
\kappa^p \int_0^\infty |\cA(\kappa s)| s^{p-1} \, ds.
\end{equation}
By Lemmas \ref{lem0409_1} and \ref{lem0409_2} it follows that
\begin{equation}
                                    \label{eq8.46}
|\cA(\kappa s)| \leq N(d,\alpha) \gamma|\cB(s)|
\end{equation}
for all $s \in (0,\infty)$.
Hence, by the Hardy-Littlewood maximal function theorem,
\begin{align*}
&\|D^2u\|_{L_p(\bR^d_T)}^p
\leq N p  \kappa^p \gamma \int_0^\infty |\cB(s)| s^{p-1} \, ds\\
&\le N\gamma \int_0^\infty\left|\left\{ (t,x) \in (-\infty,T) \times \bR^d:\gamma^{-\frac 1{ p_1}}\left( \cS\cM |D^2 u|^{p_0}(t,x)\right)^{\frac 1 {p_0}} > s/2 \right\}\right| s^{p-1} \, ds\\
&\quad +
N\gamma \int_0^\infty\left|\left\{ (t,x) \in (-\infty,T) \times \bR^d:\gamma^{-\frac 1 {p_0}}\left( \cM |f|^{p_0} (t,x) \right)^{\frac 1 {p_0}} > s/2 \right\}\right| s^{p-1} \, ds\\
&\leq N \gamma^{1-p/p_1} \|D^2u\|^p_{L_p(\bR^d_T)} + N \gamma^{1-p/p_0} \|f\|^p_{L_p(\bR^d_T)},
\end{align*}
where $N = N(d,\delta,\alpha,p)$.
Now choose $\gamma \in (0,1)$ so that
$$
N \gamma^{1-p/p_1} < 1/2,
$$
which is possible because $p\in (p_0,p_1)$.
Then we have
$$
\|D^2u\|_{L_p(\bR^d_T)} \leq N \|f\|_{L_p(\bR^d_T)}.
$$
From this and the equation, we arrive at \eqref{eq0411_04} for $p \in (p_0,p_1)$.
We repeat this procedure.
Recall \eqref{eq0411_05}, which shows that each time the increment from $p_0$ to $p_1$ can be made bigger than a positive number depending only on $d$ and $\alpha$.
Thus in finite steps, we get a $p_0$ which is larger than $d/2 + 1/\alpha$, so that $p_1=p_1(d,\alpha,p_0)=\infty$. Therefore, the theorem is proved for any $p\in [2,\infty)$.

For $p \in (1,2)$, we use a duality argument.
We only prove the a priori estimate \eqref{eq0411_04}.
Without loss of generality, assume that $u \in C_0^\infty\left([0,T] \times \bR^d\right)$ with $u(0,x) = 0$ satisfies
$$
-\partial_t^\alpha u + a^{ij} D_{ij} u = f
$$
in $(0,T) \times \bR^d$.
Let $\phi \in L_q(\bR^d_T)$, where $1/p+1/q=1$.
Then
$$
\phi(-t,x) \in L_q\left((-T,0) \times \bR^d\right).
$$
Find $w \in \bH_{q,0}^{\alpha,2}\left((-T,0) \times \bR^d\right)$ satisfying
$$
- \partial_t^\alpha w + a^{ij}(-t) D_{ij} w = \phi(-t,x)
$$
in $(-T,0) \times \bR^d$
with the estimate
$$
\|D^2w\|_{L_q\left((-T,0) \times \bR^d\right)} \leq N \|\phi(-t,x)\|_{L_q\left((-T,0) \times \bR^d\right)} = N \|\phi\|_{L_q(\bR^d_T)},
$$
where
$$
\partial_t^\alpha w = \partial_t I_{-T}^{1-\alpha} w.
$$
Considering $w_k \in C_0^\infty\left([-T,0] \times \bR^d\right)$ with $w_k(-T,0) = 0$ such that $w_k \to w$ in $\bH_{q,0}^{\alpha,2}\left((-T,0) \times \bR^d \right)$, we observe that
\begin{align*}
&\int_0^T \int_{\bR^d} \phi D^2 u  \, dx \, dt = \int_{-T}^0 \int_{\bR^d} \phi(-t,x) D^2 u(-t,x) \, dx \, dt\\
&= \int_{-T}^0 \int_{\bR^d} \left(-\partial_t^\alpha w + a^{ij}(-t) D_{ij} w \right) D^2 u(-t,x) \, dx \, dt\\
&= \int_0^T \int_{\bR^d} \left( -\partial_t^\alpha u(t,x) + a^{ij}(t)D_{ij} u(t,x) \right) D^2 w(-t,x) \, dx \, dt\\
&= \int_0^T \int_{\bR^d} f(t,x) D^2 w(-t,x) \, dx \, dt \leq N\|f\|_{L_p(\bR^d_T)} \|\phi\|_{L_q(\bR^d_T)}.
\end{align*}
It then follows that
$$
\|D^2u\|_{L_p(\bR^d_T)} \leq N \|f\|_{L_p(\bR^d_T)},
$$
from which and the equation, we finally obtain \eqref{eq0411_04}.
\end{proof}

To prove Theorem \ref{main_thm}, we extend Proposition \ref{prop0406_1} to the case when $a^{ij}=a^{ij}(t,x)$ satisfying Assumption \ref{assump2.2}.

\begin{proposition}
							\label{prop0515_1}
Let $p\in (1,\infty)$, $\alpha,\gamma_0 \in (0,1)$, $T \in (0,\infty)$, $\mu\in (1,\infty)$, $\nu=\mu/(\mu-1)$, and $a^{ij} = a^{ij}(t,x)$ satisfying Assumption \ref{assump2.2} ($\gamma_0$).
Assume that $u \in \bH_{p,0}^{\alpha,2}(\bR^d_T)$ vanishes for $x\notin B_{R_0}(x_1)$ for some $x_1\in \bR^d$, and satisfies
\eqref{eq0411_03} in $(0,T) \times \bR^d$.
Then there exists
$$
p_1 = p_1(d, \alpha,p)\in (p,\infty]
$$
satisfying \eqref{eq0411_05} and the following.
For $(t_0,x_0) \in [0,T] \times \bR^d$ and $R \in (0,\infty)$,
there exist
$$
w \in \bH_{p,0}^{\alpha,2}((t_0-R^{2/\alpha}, t_0)\times \bR^d), \quad v \in \bH_{p,0}^{\alpha,2}((S,t_0) \times \bR^d),
$$
where $S = \min\{0, t_0 - R^{2/\alpha}\}$,
such that $u = w + v$ in $Q_R(t_0,x_0)$,
$$
\left( |D^2w|^p \right)_{Q_R(t_0,x_0)}^{1/p} \le N \left( |f|^p \right)_{Q_{2R}(t_0,x_0)}^{1/p}+N\gamma_0^{1/(p\nu)}\left( |D^2 u|^{p\mu} \right)_{Q_{2R}(t_0,x_0)}^{1/(p\mu)},
$$
and
\begin{multline*}
\left( |D^2v|^{p_1} \right)_{Q_{R/2}(t_0,x_0)}^{1/p_1} \leq N \left( |f|^p \right)_{Q_{2R}(t_0,x_0)}^{1/p}+N\gamma_0^{1/(p\nu)}\left( |D^2 u|^{p\mu} \right)_{Q_{2R}(t_0,x_0)}^{1/(p\mu)}
\\
+ N \sum_{k=0}^\infty 2^{-k\alpha + 2} \left( \dashint_{\!t_0 - (2^{k+1}+1)R^{2/\alpha}}^{\,\,\,t_0} \dashint_{B_R(x_0)} |D^2u(s,y)|^p \, dy \, ds \right)^{1/p},
\end{multline*}
where $N=N(d,\delta, \alpha,p,\mu)$.
\end{proposition}
\begin{proof}
Denote
$$
Q:=\left\{
    \begin{array}{ll}
      Q_{2R}(t_0,x_0) & \hbox{when $2R\le R_0$;} \\
      (t_0 - (2R_0)^{2/\alpha}, t_0)\times B_{R_0}(x_1) & \hbox{otherwise.}
    \end{array}
  \right.
$$
Note that in both cases $|Q|\le |Q_{2R}(t_0,x_0)|$.
Thus,
by Assumption \ref{assump2.2} and Remark \ref{rem2.3}, we can find $\bar a^{ij}=\bar a^{ij}(t)$ such that
\begin{equation}
                            \label{eq8.09}
\sup_{i,j}\dashint_{Q_{2R}(t_0,x_0)}|a^{ij}-\bar a^{ij}(t)|1_Q\,dx\,dt
\le \sup_{i,j}\dashint_{Q}|a^{ij}-\bar a^{ij}(t)|\,dx\,dt\le 2\gamma_0,
\end{equation}
where $1_Q$ is the indicator function of $Q$.
We then rewrite \eqref{eq0411_03} into
$$
-\partial_t^\alpha u + \bar a^{ij}(t)D_{ij} u = \tilde f
:=f+(\bar a^{ij}(t)-a^{ij})D_{ij} u.
$$
Now that Theorem \ref{thm0412_1} holds for this equation with the same $p$,  it follows from Proposition \ref{prop0406_1} that there exist
$$
w, \, v \in \bH_p^{\alpha,2}((t_0-R^{2/\alpha}, t_0)\times \bR^d)
$$
such that $u = w + v$ in $Q_R(t_0,x_0)$, and \eqref{eq8.13}--\eqref{eq0411_01} hold with $\tilde f$ in place of $f$. To conclude the proof, it remains to notice that by H\"older's inequality and \eqref{eq8.09},
\begin{align*}
&\left( |\tilde f|^p \right)_{Q_{2R}(t_0,x_0)}^{1/p}
\le \left( |f|^p \right)_{Q_{2R}(t_0,x_0)}^{1/p}+
\left( |(\bar a^{ij}(t)-a^{ij})D_{ij} u|^p \right)_{Q_{2R}(t_0,x_0)}^{1/p}\\
&\le \left( |f|^p \right)_{Q_{2R}(t_0,x_0)}^{1/p}+
N\left( |(\bar a-a)1_{Q}|^{p\nu} \right)_{Q_{2R}(t_0,x_0)}^{1/(p\nu)}
\left( |D^2 u|^{p\mu} \right)_{Q_{2R}(t_0,x_0)}^{1/(p\mu)}\\
&\le N \left( |f|^p \right)_{Q_{2R}(t_0,x_0)}^{1/p}+N\gamma_0^{1/(p\nu)}\left( |D^2 u|^{p\mu} \right)_{Q_{2R}(t_0,x_0)}^{1/(p\mu)}.
\end{align*}
\end{proof}

Now we define $\cA(s)$ as in \eqref{eq0406_04}, but instead of \eqref{eq0406_05} we define
\begin{multline*}
\cB(s) = \big\{ (t,x) \in (-\infty,T) \times \bR^d:
\gamma^{-1/p}\left( \cM |f|^p (t,x) \right)^{1/p}
\\
+\gamma^{-1/p}\gamma_0^{1/(p\nu)}\left( \cM |D^2 u|^{p\mu} (t,x) \right)^{1/(p\mu)}
+ \gamma^{-1/p_1}\left( \cS\cM |D^2 u|^p(t,x)\right)^{1/p} > s  \big\}.
\end{multline*}

By following the proof of Lemma \ref{lem0409_1} with minor modifications, from Proposition \ref{prop0515_1}, we get the following lemma.

\begin{lemma}
							\label{lem6.3}
Let $p\in (1,\infty)$, $\alpha,\gamma_0,\gamma \in (0,1)$, $T \in (0,\infty)$, $R \in (0,\infty)$, $\mu\in (1,\infty)$, $\nu=\mu/(\mu-1)$, and $a^{ij} = a^{ij}(t,x)$ satisfying Assumption \ref{assump2.2} ($\gamma_0$).
Assume that $u \in \bH_{p,0}^{\alpha,2}(\bR^d_T)$ vanishes for $x\notin B_{R_0}(x_1)$ for some $x_1\in \bR^d$, and satisfies \eqref{eq0411_03} in $(0,T) \times \bR^d$.
Then, there exists a constant $\kappa = \kappa(d,\delta,\alpha,p,\mu) > 1$ such that the following hold:
for $(t_0,x_0) \in (-\infty,T] \times \bR^d$ and $s>0$, if
\begin{equation*}
|\cC_{R/4}(t_0,x_0) \cap \cA(\kappa s)|  \geq \gamma |\cC_{R/4}(t_0,x_0)|,
\end{equation*}
then we have
$$
\hat\cC_{R/4}(t_0,x_0) \subset \cB(s).
$$
\end{lemma}

Finally, we give the proof of Theorem \ref{main_thm}.
\begin{proof}[Proof of Theorem \ref{main_thm}]
As before we may assume that $u \in C_0^\infty\left([0,T] \times \bR^d\right)$ with $u(0,x) = 0$ and prove the a priori estimate \eqref{eq0411_04c}. We divide the proof into three steps.

{\em Step 1.} We assume that $u$ vanishes for $x\notin B_{R_0}(x_1)$ for some $x_1\in \bR^d$, and $b\equiv c\equiv 0$. We take $p_0\in (1,p)$ and $\mu\in (1,\infty)$ depending only on $p$ such that $p_0<p_0\mu<p<p_1$, where $p_1=p_1(d,\alpha,p_0)$ is taken from Proposition \ref{prop0515_1}. By Lemmas \ref{lem6.3} and \ref{lem0409_2}, we have \eqref{eq8.46}, which together with \eqref{eq8.47} and the Hardy-Littlewood maximal function theorem implies that
{\small \begin{align*}
&\|D^2u\|_{L_p(\bR^d_T)}^p
\leq N p  \kappa^p \gamma \int_0^\infty |\cB(s)| s^{p-1} \, ds\\
&\le N\gamma \int_0^\infty\left|\left\{ (t,x) \in (-\infty,T) \times \bR^d:\gamma^{-\frac 1 {p_1}}\left( \cS\cM |D^2 u|^{p_0}(t,x)\right)^{\frac 1 {p_0}} > s/3 \right\}\right| s^{p-1} \, ds\\
&\quad+
N\gamma \int_0^\infty\left|\left\{ (t,x) \in (-\infty,T) \times \bR^d:\gamma^{-\frac 1 {p_0}}\left( \cM |f|^{p_0} (t,x) \right)^{\frac 1 {p_0}} > s/3 \right\}\right| s^{p-1} \, ds\\
&\quad+
N\gamma \int_0^\infty\left|\left\{ (t,x) \in (-\infty,T) \times \bR^d:\gamma^{-\frac 1 {p_0}}\gamma_0^{\frac 1 {p_0\nu}}
\left( \cM |D^2 u|^{p_0\mu} (t,x) \right)^{\frac 1 {p_0\mu}} > s/3 \right\}\right| s^{p-1} \, ds\\
&\leq N (\gamma^{1-p/p_1}+\gamma^{1-p/p_0}\gamma_0^{p/(p_0\nu)}) \|D^2u\|^p_{L_p(\bR^d_T)} + N \gamma^{1-p/p_0} \|f\|^p_{L_p(\bR^d_T)},
\end{align*}}
where $N = N(d,\delta,\alpha,p)$.
Now choose $\gamma \in (0,1)$ sufficiently small and then $\gamma_0$ sufficiently small, depending only on $d$, $\delta$, $\alpha$, and $p$, so that
$$
N (\gamma^{1-p/p_1}+\gamma^{1-p/p_0}\gamma_0^{p/(p_0\nu)}) < 1/2.
$$
Then we have
$$
\|D^2u\|_{L_p(\bR^d_T)} \leq N(d,\delta,\alpha,p) \|f\|_{L_p(\bR^d_T)}.
$$
From this and the equation, we arrive at \eqref{eq0411_04}.

{\em Step 2.} In this step, we show that under the assumptions of the theorem with $\gamma_0$ being the constant from the previous step, we have
\begin{equation}
							\label{eq8.59}
\|\partial_t^\alpha u\|_{L_p(\bR^d_T)} + \|D^2 u\|_{L_p(\bR^d_T)} \leq N \|f\|_{L_p(\bR^d_T)}+N_1\|u\|_{L_p(\bR^d_T)},
\end{equation}
where $N=N(d,\delta,\alpha,p)$ and $N_1=N_1(d,\delta,\alpha,p,R_0)$. By moving the lower-order terms to the right-hand side of the equation, and using interpolation inequalities, without loss of generality, we may assume that $b\equiv c\equiv 0$.
Now \eqref{eq8.59} follows a standard partition of unity argument with respect to $x$ and interpolation inequalities.

{\em Step 3.} In this step, we show how to get rid of the second term on the right-hand side of \eqref{eq8.59} and conclude the proof of \eqref{eq0411_04c}. By \eqref{eq8.59} and Lemma \ref{lem1110_1}, we can find $q\in (p,\infty)$, depending on $\alpha$ and $p$, such that for any $T'\in (0,T]$,
\begin{align*}
\|u\|_{L_p\left(\bR^d; L_q(0,T')\right)}
&\le N(\alpha,p,T)\|\partial^t_\alpha u\|_{L_p((0,T');L_p(\bR^d))}\\
&\le N\|f\|_{L_p(\bR^d_{T'})}+N_1\|u\|_{L_p(\bR^d_{T'})},
\end{align*}
where $N=N(d,\delta,\alpha,p,T)$ and $N_1=N_1(d,\delta,\alpha,p,T,R_0)$. Next we take a sufficiently large integer $m=m(d,\delta,\alpha,p,T,R_0)$ such that $N_1(T/m)^{1/p-1/q}\le 1/2$. Then for any $j=0,2,\ldots,m-1$, by H\"older's inequality and the above inequality with $T'=(j+1)T/m$, we have
\begin{align*}
&\|u\|_{L_p((jT/m,(j+1)T/m);L_p(\bR^d))}
\le (T/m)^{1/p-1/q}\|u\|_{L_p\left(\bR^d;L_q(jT/m,(j+1)T/m)\right)}\\
&\le N\|f\|_{L_p(\bR^d_T)}+\frac 1 2\|u\|_{L_p((0,(j+1)T/m);L_p(\bR^d))}.
\end{align*}
This implies that
$$
\|u\|_{L_p((jT/m,(j+1)T/m);L_p(\bR^d))}
\le N\|f\|_{L_p(\bR^d_T)}+\|u\|_{L_p((0,jT/m);L_p(\bR^d))}.
$$
By an induction on $j$, we obtain
$$
\|u\|_{L_p(\bR^d_T)}\le N\|f\|_{L_p(\bR^d_T)},
$$
which together with \eqref{eq8.59} yields \eqref{eq0411_04c}. The theorem is proved.
\end{proof}

\section*{Acknowledgment}
The authors would like to thank Nicolai V. Krylov for telling us a simple proof of \eqref{eq0904_01}, and the referee for helpful comments.
The authors also thank Kyeong-hun Kim for bringing our attention to the problems discussed in this paper.

\appendix

\section{Sobolev embeddings for $\bH_{p,0}^{\alpha,2}$ and a ``crawling of ink spots'' lemma}

In the proof of Lemma \ref{lem0123_1} as well as in several places of this paper, we use the following properties of the operator $I^\alpha$.
In the sequel, let $T\in (0,\infty)$ be a constant.

\begin{lemma}
							\label{lem1115_1}
Let
$p \in (1,\infty)$, $q \in (1,\infty)$, and $\alpha \in (0, 1/p)$ satisfy
$$
q > p, \quad \alpha - 1/p = - 1/q.
$$
Then we have
$$
\|I^\alpha \psi\|_{L_q(0,T)} \le N(\alpha,p)\|\psi\|_{L_p(0,T)}
$$
for $\psi \in L_p(0,T)$.
\end{lemma}

\begin{proof}
See \cite[Theorem 4]{MR1544927}.
\end{proof}

\begin{lemma}
							\label{lem1018_01}
Let $\alpha \in (0,1)$, $\psi \in L_p(0,T)$, and
$p \in [1,\infty]$ and $q \in [1,\infty]$ satisfy
\begin{equation*}
\alpha - 1/p > - 1/q.
\end{equation*}
Then we have
\begin{equation*}
\|I^\alpha \psi\|_{L_q(0,T)} \le N(\alpha,p,q)T^{\alpha-1/p+1/q} \|\psi\|_{L_p(0,T)}.
\end{equation*}
\end{lemma}

\begin{proof}
First, consider $p=1$.
In this case, $q \in [1,1/(1-\alpha))$.
Then
\begin{align*}
&\Gamma(\alpha)|I^\alpha \psi(t)| \le \int_0^t (t-s)^{\alpha-1} |\psi(s)| \, ds
= \int_0^t (t-s)^{\alpha-1} |\psi(s)|^{\frac{1}{q}} |\psi(s)|^{\frac{q-1}{q}} \, ds\\
&\le \left(\int_0^t (t-s)^{(\alpha-1)q}|\psi(s)| \, ds\right)^{\frac{1}{q}} \left( \int_0^t |\psi(s)| \, ds \right)^{\frac{q-1}{q}}.
\end{align*}
Thus,
\begin{align*}
&\| I^{\alpha} \psi(t)\|_{L_q(0,T)} \leq N(\alpha) \|\psi\|_{L_1(0,T)}^{1-\frac{1}{q}}
\left(\int_0^T \int_0^t (t-s)^{(\alpha-1)q}|\psi(s)| \, ds \, dt\right)^{\frac{1}{q}}\\
&\leq N(\alpha,q) T^{\alpha -1 + 1/q} \|\psi\|_{L_1(0,T)},
\end{align*}
where we used the condition that $(\alpha-1)q > -1$.

If $p \in (1,\infty]$, $q \in [1,\infty)$, and $\alpha - 1/p > -1/q$, then one can find $p_1 \in (1,p]$, $q_1 \in [q,\infty)$ such that $\alpha - 1/p_1 = -1/q_1$.
The result then follows from  Lemma \ref{lem1115_1} and H\"{o}lder's inequality.

Finally, if
$$
p \in (1,\infty], \quad q = \infty, \quad \alpha - 1/p > 0,
$$
then
\begin{align*}
&\Gamma(\alpha)|I^\alpha \psi(t)| \le \int_0^t (t-s)^{\alpha-1} |\psi(s)| \, ds\\
&\le \left(\int_0^t (t-s)^{(\alpha-1)\frac{p}{p-1}} \, ds\right)^{\frac{p-1}{p}} \left( \int_0^t |\psi(s)|^p \, ds\right)^{\frac{1}{p}}\\
&\le T^{\frac{\alpha p - 1}{p}} \left(\frac{p-1}{\alpha p - 1}\right)^{1-1/p} \|\psi\|_{L_p(0,T)},
\end{align*}
where we again use the condition that $(\alpha-1)p/(p-1) > -1$.
The lemma is proved.
\end{proof}

\begin{remark}
							\label{rem0120_1}
From Lemma \ref{lem1018_01}, if $u = u(t,x) \in L_p(\Omega_T)$, $1\leq p \leq \infty$, then
$I^\alpha u \in L_p(\Omega_T)$ and
$$
\|I^\alpha u \|_{L_p(\Omega_T)} \leq N(\alpha, p) T^\alpha \|u\|_{L_p(\Omega_T)}.
$$
\end{remark}

\begin{lemma}
							\label{lem1115_3}
Let $\psi \in C^1([0,T])$ and $\psi(0) = 0$.
Then
$$
I^\alpha D_t^\alpha \psi
= I^\alpha \partial_t^\alpha \psi = \psi.
$$
\end{lemma}

\begin{proof}
Since $\psi(0) = 0$, we have
$$
D_t^\alpha \psi(t) = \partial_t^\alpha \psi(t) = \frac{1}{\Gamma(1-\alpha)}\int_0^t (t-s)^{-\alpha} \psi'(s) \, ds.
$$
Then
\begin{align*}
&I^\alpha D_t^\alpha \psi (t) = \frac{1}{\Gamma(\alpha)}\frac{1}{\Gamma(1-\alpha)}\int_0^t (t-s)^{\alpha-1} \int_0^s (s-r)^{-\alpha}\psi'(r) \, dr \, ds\\
&=\frac{1}{\Gamma(\alpha)}\frac{1}{\Gamma(1-\alpha)} \int_0^t \psi'(r) \int_r^t (t-s)^{\alpha-1}(s-r)^{-\alpha} \, ds \, dr = \int_0^t \psi'(r) \, dr = \psi(t).
\end{align*}
\end{proof}

\begin{lemma}
							\label{lem1115_2}
Let $p, q \in (1,\infty)$ and $\alpha \in (0,1/p)$ satisfy
$$
\alpha - 1/p = - 1/q.
$$
Then
$$
\|\psi\|_{L_q(0,T)} \le N(\alpha,p)\|\partial_t^\alpha \psi\|_{L_p(0,T)}
$$
for $\psi \in C^1([0,T])$ such that $\psi(0) = 0$.
\end{lemma}

\begin{proof}
Using Lemmas \ref{lem1115_3} and \ref{lem1115_1}, we obtain that
$$
\|\psi\|_{L_q(0,T)} = \|I^\alpha \partial_t^\alpha \psi\|_{L_q(0,T)} \le N(\alpha,p)\|\partial_t^\alpha \psi\|_{L_p(0,T)}.
$$
\end{proof}

\begin{lemma}
							\label{lem1110_1}
Let $\psi \in C^1([0,T])$ such that $\psi(0) = 0$.
Then
$$
\|\psi\|_{L_q(0,T)} \le N(\alpha,p,q)T^{\alpha-1/p+1/q}\|\partial_t^\alpha \psi\|_{L_p(0,T)},
$$
where $p \in [1,\infty]$, $q \in [1,\infty]$, and
$$
\alpha - 1/p > - 1/q.
$$
\end{lemma}

\begin{proof}
We have
$$
\|\psi\|_{L_q(0,T)} = \|I^\alpha \partial_t^\alpha \psi\|_{L_q(0,T)} \le N(\alpha,p,q)T^{\alpha-1/p+1/q} \|\partial_t^\alpha \psi\|_{L_p(0,T)},
$$
where the second inequality is due to Lemma \ref{lem1018_01}.
\end{proof}

\begin{lemma}[Multiplicative inequality]
							\label{lem1115_4}
Let $p, q, r \in (1,\infty)$ and $\alpha \in (0,1/p)$.
Let $\psi \in C^1([0,T])$ such that $\psi(0) = 0$.
Then
\begin{equation*}
\|\psi\|_{L_q(0,T)} \le N(\alpha,p,\theta) \|\partial_t^\alpha \psi\|_{L_p(0,T)}^{\theta} \|\psi\|_{L_r(0,T)}^{1-\theta}
\end{equation*}
for all $\theta \in [0,1]$ satisfying
\begin{equation}
							\label{eq1115_02}
1/q = \left( 1/p - \alpha \right) \theta + (1-\theta)/r.
\end{equation}
\end{lemma}

\begin{proof}
By Lemma \ref{lem1115_2}, we can clearly assume that $\theta \in (0,1)$.
Under the conditions $\alpha < 1/p$ and \eqref{eq1115_02}, we see that
$$
\frac{(1-\theta)q}{r} < 1.
$$
Note that by H\"older's inequality,
$$
\|\psi\|_{L_q(0,T)} \leq \left(\int_0^T |\psi|^r \, dt \right)^{\frac{(1-\theta)}{r}}\left(\int_0^T |\psi|^{\theta q A'} \, dt \right)^{\frac{1}{qA'}},
$$
where $A'$ satisfies
$$
\frac{(1-\theta)q}{r} + \frac{1}{A'}=1.
$$
Hence, by Lemma \ref{lem1115_2} and the fact that
$$
\alpha < 1/p, \quad \theta q A' > 1, \quad  \alpha - \frac{1}{p} = - \frac{1}{\theta q A'},
$$
it follows
$$
\|\psi\|_{L_q(0,T)}
\leq \|\psi\|_{L_r(0,T)}^{(1-\theta)}
\| \psi\|_{L_{\theta q A'}(0,T)}^{\theta}
\leq N(\alpha, p)^{\theta} \|\psi\|_{L_r(0,T)}^{(1-\theta)} \|\partial_t^\alpha \psi\|_{L_p(0,T)}^{\theta}.
$$
The lemma is proved.
\end{proof}

\begin{theorem}[Embedding with $\alpha$-time derivative and $2$-spatial derivatives with $p < \min\{1/\alpha,d/2\}$]
							\label{thm1204_1}
Let $\alpha \in (0, 1)$ and $p, q \in (1,\infty)$ satisfy
$$
p < \min\{1/\alpha, d/2\}, \quad p < q < q^* := \frac{1/\alpha + d/2}{1/(\alpha p) + d/(2p) - 1}.
$$
Then
\begin{equation}
							\label{eq0213_01}
\|\psi\|_{L_q(\bR^d_T)} \le N  \|D_x^2 \psi\|_{L_p(\bR^d_T)}^{\theta} \|\partial_t^\alpha \psi \|_{L_p(\bR^d_T)}^{\tau(1-\theta)} \|\psi\|_{L_p(\bR^d_T)}^{(1-\tau)(1-\theta)}
\end{equation}
for $\psi \in \bH_{p,0}^{\alpha,2}(\bR^d_T)$, where
$$
\theta = \frac{d}{2}\left(\frac{1}{p} - \frac{1}{q}\right) \in (0,1), \quad \tau = \frac{2}{\alpha d} \frac{\theta}{1-\theta} \in (0,1),
$$
and
$N = N(d,\alpha,p,q)$, but independent of $T$.
If $q = q^*$, then
\begin{equation}
							\label{eq0213_02}
\|\psi\|_{L_q(\bR^d_T)} \le N  \|D_x^2 \psi\|_{L_p(\bR^d_T)}^{\alpha d/(2+\alpha d)} \|\partial_t^\alpha \psi \|_{L_p(\bR^d_T)}^{2/(2+\alpha d)}.
\end{equation}
\end{theorem}

\begin{proof}
By the definition of $\bH_{p,0}^{\alpha,2}(\bR^d_T)$, we may assume that $\psi \in C_0^\infty\left([0,T] \times \bR^d \right)$ and $\psi(0,x) = 0$.
By the Sobolev embedding in $x$, we have
\begin{equation}
                                    \label{eq11.55}
\|\psi\|_{L_p((0,T);L_{pd/(d-2p)}(\bR^d))} \le N\|D_x^2 \psi\|_{L_p(\bR^d_T)}.
\end{equation}
Similarly, by Lemma \ref{lem1115_4} with $\theta = 1$, we have
$$
\|\psi\|_{L_p(\bR^d;L_{p/(1-\alpha p)}((0,T)))} \le N\|\partial_t^\alpha \psi\|_{L_p(\bR^d_T)},
$$
which together with the Minkowski inequality implies that
\begin{equation}
                               \label{eq11.56}
\|\psi\|_{L_{p/(1-\alpha p)}((0,T);L_p(\bR^d))}
= \left\| \int_{\bR^d} |\psi(\cdot,x)|^p \, dx \right\|_{L_{\frac{1}{1-\alpha p}}(0,T)}^{\frac{1}{p}}
\le N\|\partial_t^\alpha \psi\|_{L_p(\bR^d_T)}.
\end{equation}
By \eqref{eq11.55}, \eqref{eq11.56}, and H\"older's inequality, we immediately get \eqref{eq0213_02}.
Finally, \eqref{eq0213_01} follows from  \eqref{eq0213_02} and H\"older's inequality.
\end{proof}

From Theorem \ref{thm1204_1} the following corollary follows easily.

\begin{corollary}
							\label{cor1211_1}
Let $\alpha \in (0, 1)$ and $p, q \in (1,\infty)$ satisfy
$$
p < \min\{1/\alpha,d/2\}, \quad p<q \leq q^* := \frac{1/\alpha + d/2}{1/(\alpha p) + d/(2p) - 1}.
$$
Then we have
$$
\|\psi\|_{L_q\left((0,T) \times B_1\right)} \le N \|\psi\|_{\bH_p^{\alpha,2}\left((0,T) \times B_1\right)}
$$
for any $\psi \in \bH_{p,0}^{\alpha,2}\left((0,T) \times B_1\right)$, where
$N = N(d,\alpha,p,q)$, but independent of $T$.
If $p \leq d/2$ and $p \leq 1/\alpha$, then the same estimate holds for $q\in [1,q^*)$ with $N$ depending also on $T$.
\end{corollary}

\begin{proof}
If $p < d/2$ and $p < 1/\alpha$, the result follows easily from Theorem \ref{thm1204_1} with an extension of $\psi$ to a function in $\bH_{p,0}^{\alpha,2}(\bR^d_T)$.
If $p = d/2$ or $p = 1/\alpha$, then find $\varepsilon > 0$ such that
$$
q < \frac{1/\alpha + d/2}{1/(\alpha (p-\varepsilon)) + d/(2(p-\varepsilon)) - 1} < \frac{1/\alpha + d/2}{1/(\alpha p) + d/(2p) - 1}.
$$
Then
$$
\|\psi\|_{L_q\left((0,T) \times B_1\right)} \leq N \|\psi\|_{\bH_{p-\varepsilon}^{\alpha,2}\left((0,T) \times B_1\right)} \leq N \|\psi\|_{\bH_p^{\alpha,2}\left((0,T) \times B_1\right)}.
$$
The corollary is proved.
\end{proof}

\begin{theorem}[Embedding with $\alpha$-time derivative and $2$-spatial derivatives with $d/2 < p < 1/\alpha$]
							\label{thm1207_2}
Let $\alpha \in (0, 1)$ and $p, q \in (1,\infty)$ satisfy
$$
\frac{d}{2} < p < \frac{1}{\alpha}, \quad p < q \leq p(\alpha p + 1).
$$
Then, for $\psi \in \bH_{p,0}^{\alpha, 2}\left((0,T) \times B_1\right)$, we have
\begin{equation}
							\label{eq0213_04}
\|\psi\|_{L_q((0,T)\times B_1)} \le N \left( \sum_{0 \leq |\beta| \leq 2} \|D_x^\beta \psi\|_{L_p((0,T)\times B_1)} \right)^{1-\theta} \|\partial_t^\alpha\psi\|_{L_p((0,T)\times B_1)}^\theta,
\end{equation}
where $N = N(d,\alpha,p,q)$, but independent of $T$, and
$$
\theta = \frac{1}{\alpha} \left( \frac{1}{p} - \frac{1}{q} \right) \in (0,1).
$$
If $d/2 < p \leq 1/\alpha$, then the same estimate holds for $q$ satisfying
$$
1 \leq q < p(\alpha p + 1)
$$
with $N$ depending also on $T$.
\end{theorem}

\begin{proof}
As above, we assume that $\psi \in  \bH_{p,0}^{\alpha,2}\left((0,T) \times B_1\right) \cap C^\infty\big([0,T] \times B_1\big)$ and $\psi(0,x) = 0$.
Since $p>d/2$, by the Sobolev embedding in $x$, we have
\begin{equation}
                                    \label{eq12.01}
\|\psi\|_{L_p((0,T);L_{\infty}(B_1))} \le N\left( \sum_{0 \leq |\beta| \leq 2} \|D_x^\beta \psi\|_{L_p((0,T)\times B_1)} \right).
\end{equation}
Similarly, by Lemma \ref{lem1115_4} with $\theta = 1$ and the Minkowski inequality, we have
\begin{multline}
                               \label{eq11.56b}
\|\psi\|_{L_{p/(1-\alpha p)}((0,T);L_p(B_1))} = \left\| \int_{B_1} |\psi(\cdot,x)|^p \, dx \right\|_{L_{1/(1-\alpha p)}((0,T))}^{1/p}
\\
\leq \|\psi\|_{L_p(B_1; L_{p/(1-\alpha p)}((0,T)))} \le N\|\partial_t^\alpha \psi\|_{L_p((0,T)\times B_1)}.
\end{multline}
By \eqref{eq12.01}, \eqref{eq11.56b}, and H\"older's inequality, we immediately get \eqref{eq0213_04} with $q=p(\alpha p+1)$ and $\theta=1/(\alpha p+1)$. The general case then follows from  H\"older's inequality.
\end{proof}

\begin{theorem}[Embedding with $\alpha$-time derivative and $2$-spatial derivatives with $1/\alpha < p <  d/2$]
Let $\alpha \in (0,1)$ and $p, q \in (1,\infty)$ such that
\begin{equation*}
\frac{1}{\alpha} < p < \frac{d}{2}, \quad p < q \leq p + \frac{2p^2}{d}.
\end{equation*}
Then, for $\psi \in \bH_{p,0}^{\alpha,2}(\bR^d_T)$,
$$
\|\psi\|_{L_q(\bR^d_T)} \le N T^{\alpha\left(1-\frac{p}{q}\right)-\frac{1}{p}+\frac{1}{q}} \|\partial_t^\alpha \psi\|_{L_p(\bR^d_T)}^{1-p/q} \|D_x^2 \psi\|_{L_p(\bR^d_T)}^{\theta p/q} \|\psi\|_{L_p(\bR^d_T)}^{(1-\theta)p/q},
$$
where $N=N(d,\alpha,p,q)$ and
$
\theta = d(q-p)/(2p^2) \in (0,1]$.
\end{theorem}
\begin{proof}
As above, we assume that $\psi \in C_0^\infty\left([0,T] \times \bR^d \right)$ and $\psi(0,x) = 0$.
Since $\alpha>1/p$, by Lemma \ref{lem1110_1} and the Minkowski inequality, we have
\begin{equation}
                                    \label{eq12.14}
\|\psi\|_{L_{\infty}((0,T);L_p(\bR^d_T))}
\le \|\psi\|_{L_p(\bR^d;L_{\infty}((0,T)))}\le NT^{\alpha-1/p}\|\partial_t^\alpha \psi\|_{L_p(\bR^d_T)}.
\end{equation}
By the Sobolev embedding in $x$, we have
\begin{equation}
                                    \label{eq12.19}
\|\psi\|_{L_p((0,T);L_{dp/(d-2p)}(\bR^{d}))}\le N\|D_x^2 \psi\|_{L_p(\bR^d_T)}.
\end{equation}
By \eqref{eq12.14}, \eqref{eq12.19}, and H\"older's inequality, we get the desired estimate with $q=p+2p^2/d$. The general case then follows from H\"older's inequality.
\end{proof}

By extending $\psi \in \bH_{p,0}^{\alpha,2}\left((0,T) \times B_1 \right)$ to a function in $\bH_{p,0}^{\alpha,2}(\bR^d_T)$ and using the above theorem, we get

\begin{corollary}
[Embedding with $\alpha$-time derivative and $2$-spatial derivatives with $1/\alpha < p <  d/2$]
							\label{cor0225_1}
Let $\alpha \in (0,1)$ and $p, q \in (1,\infty)$ such that
\begin{equation*}
\frac{1}{\alpha} < p < \frac{d}{2}, \quad p < q \leq p + \frac{2p^2}{d}.
\end{equation*}
Then, for $\psi \in \bH_{p,0}^{\alpha,2}\left((0,T) \times B_1\right)$,
$$
\|\psi\|_{L_q\left((0,T) \times B_1\right)} \le N T^{\alpha\left(1-\frac{p}{q}\right)-\frac{1}{p}+\frac{1}{q}} \|\psi\|_{\bH_p^{\alpha,2}\left((0,T) \times B_1\right)},
$$
where $N=N(d,\alpha,p,q)$.
If $1/\alpha < p \leq d/2$, the same estimate holds for $q$ satisfying
$$
1 \leq q < p + 2p^2/d
$$
with $N$ depending also on $T$.
\end{corollary}

\begin{theorem}[Embedding with $\alpha$-time derivative and $2$-spatial derivatives with $\max\{1/\alpha ,d/2\} < p \leq d/2 + 1/\alpha$]
							\label{thm0214_1}
Let $\alpha \in (0,1)$ and $p, q \in (1,\infty)$ such that
\begin{equation*}
\max\{1/\alpha ,d/2\} < p \leq d/2 + 1/\alpha, \quad p < q \le 2p.
\end{equation*}
Then, for $\psi \in \bH_{p,0}^{\alpha,2}\left((0,T) \times B_1\right)$,
\begin{align*}
&\|\psi\|_{L_q\left((0,T) \times B_1\right)} \\
&\le N T^{\frac{\alpha p}{q} - \frac{1}{p} + \frac{1}{q}} \left( \sum_{0 \leq |\beta|\leq 2}\|D_x^\beta \psi\|_{L_p\left((0,T) \times B_1\right)}\right)^{1-\theta} \|\partial_t^\alpha \psi\|_{L_p\left((0,T) \times B_1\right)}^\theta,
\end{align*}
where $N=N(d,\alpha,p,q)$ and
$
\theta = p/q \in (0,1)$.
\end{theorem}

\begin{proof}
Again we assume that $\psi \in \bH_{p,0}^{\alpha,2}\left((0,T)\times B_1\right) \cap C^\infty\left([0,T] \times B_1 \right)$ and $\psi(0,x) = 0$.
We set $q':=p^2/(2p-q)\in (p,\infty]$. Since $\alpha - 1/p > 0$, from Lemma \ref{lem1110_1} and the Minkowski inequality,
$$
\|\psi\|_{L_{q'}((0,T);L_p(B_1))}
\le \|\psi\|_{L_p(B_1;L_{q'}((0,T)))}
\le N T^{\alpha - 1/p + 1/q'} \|\partial_t^\alpha \psi\|_{L_p((0,T)\times B_1)},
$$
where $N = N(\alpha,p,q')$.
This, \eqref{eq12.01}, and H\"older's inequality yield the desired inequality.
\end{proof}

\begin{lemma}[Embedding with $\alpha > 1/p$ in time]
							\label{lem0217_2}
Let $p \in (1, \infty]$, $\alpha > 1/p$, and $\psi \in C^1([0,T])$ with $\psi(0) = 0$.
Then
$$
|\psi(t_2) - \psi(t_1)| \leq N(\alpha,p) (t_2 - t_1)^{\alpha-1/p} \|\partial_t^\alpha\psi\|_{L_p(0,T)}
$$
for $0 \leq t_1 < t_2 \leq T$.
\end{lemma}

\begin{proof}
Note that
$$
\psi(t_2) - \psi(t_1) = (I^\alpha \partial_t^\alpha \psi)(t_2) - (I^\alpha \partial_t^\alpha \psi)(t_1).
$$
Set $\partial_t^\alpha \psi(t) = \Psi(t)$.
Then
\begin{align*}
&\Gamma(\alpha) \left(\psi(t_2) - \psi(t_1)\right) = \int_0^{t_2} (t_2-s)^{\alpha-1} \Psi(s) \, ds - \int_0^{t_1} (t_1 - s)^{\alpha-1}\Psi(s) \, ds\\
&= \int_0^{t_1} (t_2-s)^{\alpha-1} \Psi(s) \, ds - \int_0^{t_1} (t_1 - s)^{\alpha-1}\Psi(s) \, ds + \int_{t_1}^{t_2} (t_2-s)^{\alpha-1} \Psi(s) \, ds\\
&:= J_1 + J_2 + J_3.
\end{align*}
Note that
\begin{align*}
&J_1 + J_2
= \int_0^{t_1} \left( (t_2-s)^{\alpha-1} - (t_1-s)^{\alpha-1} \right) \Psi(s) \, ds\\
&\leq \left(\int_0^{t_1} \left| (t_2-s)^{\alpha-1} - (t_1-s)^{\alpha-1} \right|^{p/(p-1)} \, ds\right)^{(p-1)/p} \left(\int_0^{t_1} |\Psi(s)|^p \, ds\right)^{1/p},
\end{align*}
where
\begin{align*}
&\int_0^{t_1} \left| (t_2-s)^{\alpha-1} - (t_1-s)^{\alpha-1} \right|^{p/(p-1)} \, ds\\
&= \int_0^{t_1} \left[
(t_1-s)^{\alpha-1} - (t_2-s)^{\alpha-1} \right]^{p/(p-1)} \, ds =: K_1.
\end{align*}
If $2 t_1 \leq t_2$, since $(\alpha-1)\frac{p}{p-1}>-1$, it follows that
$$
K_1 \leq \int_0^{t_1} (t_1-s)^{(\alpha-1)\frac{p}{p-1}} \, ds \leq N(\alpha, p)t_1^{(\alpha-1)\frac{p}{p-1}+1} \leq N(\alpha,p)(t_2 - t_1)^{(\alpha-1)\frac{p}{p-1}+1}.
$$
If $2t_1 > t_2$, that is, $2t_1 - t_2 >0$, then
\begin{align*}
&K_1 = \int_0^{2t_1-t_2} \left[
(t_1-s)^{\alpha-1} - (t_2-s)^{\alpha-1} \right]^{p/(p-1)} \, ds\\
&\qquad + \int_{2t_1-t_2}^{t_1} \left[
(t_1-s)^{\alpha-1} - (t_2-s)^{\alpha-1} \right]^{p/(p-1)} \, ds\\
&\leq (1-\alpha)^{\frac{p}{p-1}} (t_2-t_1)^{\frac{p}{p-1}}\int_0^{2t_1-t_2}  (t_1-s)^{(\alpha-2)\frac{p}{p-1}} \, ds + \int_{2t_1-t_2}^{t_1} (t_1-s)^{(\alpha-1)\frac{p}{p-1}} \, ds\\
&= N (t_2-t_1)^{\frac{p}{p-1}} \left[ (t_2-t_1)^{(\alpha-2)\frac{p}{p-1} + 1} - t_1^{(\alpha-2)\frac{p}{p-1} + 1}\right] + N (t_2-t_1)^{(\alpha-1)\frac{p}{p-1} + 1}\\
&\leq N(\alpha,p) (t_2-t_1)^{(\alpha-1)\frac{p}{p-1} + 1},
\end{align*}
where $N = N(\alpha,p)$ and we used the fact that
$$
(\alpha-2)\frac{p}{p-1} + 1 < 0.
$$
Hence,
$$
J_1 + J_2 \leq N(\alpha,p) (t_2-t_1)^{\alpha - 1/p} \|\Psi\|_{L_p(0,T)}.
$$
For the term $I_3$, we see that
$$
J_3 \leq \left(\int_{t_1}^{t_2} (t_2-s)^{(\alpha-1)\frac{p}{p-1}} \,ds\right)^{\frac{p-1}{p}} \|\Psi\|_{L_p(0,T)} \leq N(\alpha,p) (t_2-t_1)^{\alpha-1/p}\|\Psi\|_{L_p(0,T)}.
$$
Therefore,
\begin{align*}
&|\psi(t_2,x)-\psi(t_1,x)| \leq N(\alpha,p) (t_2-t_1)^{\alpha-1/p}\|\Psi\|_{L_p(0,T)}\\
&= N(\alpha,p) (t_2-t_1)^{\alpha-1/p}\|\partial_t^\alpha \psi\|_{L_p(0,T)}.
\end{align*}
\end{proof}

Recall that
$$
Q_R(t,x) =Q_{R,R}(t,x) = (t-R^{2/\alpha}, t) \times B_R(x).
$$
For the H\"{o}lder semi-norm, we denote
$$
[u]_{C^{\sigma_1, \sigma_2}(\cD)} = \sup_{\substack{(t,x),(s,y) \in \cD \\ (t,x) \neq (s,y)}}\frac{|u(t,x) - u(s,y)|}{|t-s|^{\sigma_1} + |x-y|^{\sigma_2}},
$$
where $\cD \subset \bR \times \bR^d$.

\begin{lemma}[Embedding with $2$-spatial derivatives with $p\in(d/2+1/\alpha,d+2/\alpha)$]
							\label{lem0225_1}
Let $\alpha \in (0,1)$ and $p \in (1,\infty)$ such that
$$
\sigma := 2-(d+2/\alpha)/p \in (0,1).
$$
Assume that $\psi \in C^\infty\big(\overline{(0,1) \times B_1}\big)$ and $\psi(0,x) = 0$.
For any $\varepsilon \in (0,1/2)$ and
$$
(t,x), (t,y) \in (0,1) \times B_1,
$$
we have
\begin{multline}
							\label{eq0226_01}
|\psi(t,x)-\psi(t,y)|
 \leq (2^{\sigma-1}  + 3 \varepsilon^\sigma) |x-y|^\sigma K
\\
+ N \varepsilon^{-2/(\alpha p) - d/p + 1/p} |x-y|^\sigma \|D_x^2\psi \|_{L_p\left((0,1) \times B_1 \right)},
\end{multline}
provided that either $B_h(x) \subset B_1$ or $B_h(y) \subset B_1$, $h := |x-y|$, where $N=N(d,\alpha,p)$ and
\begin{equation}
							\label{eq0225_02}
K = \sup_{\substack{(t,x),(s,y) \in (0,1)\times B_1 \\ (t,x) \neq (s,y)}}\frac{|\psi(t,x) - \psi(s,y)|}{|t-s|^{\sigma \alpha/2} + |x-y|^\sigma}.
\end{equation}
\end{lemma}

\begin{proof}
Without loss of generality, we assume that $B_h(x) \subset B_1$.
Due to an appropriate orthogonal transformation, we assume  that
$$
x=(x_1,x'), \quad y = (x_1-h,x').
$$
Since $B_h(x) \subset B_1$, we have
$$
(x_1+h,x') \in \overline{B_1}.
$$
For any $\varepsilon \in (0,1/2)$, set
$$
\rho = \varepsilon h.
$$
We write
\begin{align*}
&\psi(t,x_1,x')-\psi(t,x_1-h,x')
= \frac{1}{2}\left[ \psi(t,x_1 + h,x') - \psi(t,x_1-h,x') \right]\\
&\qquad - \frac{1}{2} \left[ \psi(t,x_1+h,x') - 2 \psi(t,x_1,x') + \psi(t,x_1-h,x') \right].
\end{align*}
Thus,
\begin{multline}
							\label{eq0225_03}
|\psi(t,x)-\psi(t,y)| \leq \frac{1}{2} (2h)^\sigma K
\\
+ \frac{1}{2} \left|\psi(t,x_1+h,x') - 2 \psi(t,x_1,x') + \psi(t,x_1-h,x')\right|.
\end{multline}
To estimate the last term in the above inequalities,
we observe that
\begin{align}
							\label{eq0225_04}
&\left|\psi(t,x_1+h,x') - 2 \psi(t,x_1,x') + \psi(t,x_1-h,x')\right|\nonumber
\\
&\leq |\psi(t,x_1+h-\rho,x')-2\psi(t,x_1,x')+\psi(t,x_1-h+\rho,x')|\nonumber
\\
&\quad + |\psi(t,x_1+h,x') - \psi(t,x_1+h-\rho,x')| + |\psi(t,x_1-h,x') - \psi(t,x_1-h+\rho,x')|\nonumber
\\
&\leq 2 K \rho^\sigma + |\psi(t,x_1+h-\rho,x')-2\psi(t,x_1,x')+\psi(t,x_1-h+\rho,x')|\nonumber
\\
&:= 2K \rho^\sigma + J.
\end{align}
We consider
$$
\big((t-\rho^{2/\alpha}, t+\rho^{2/\alpha}) \cap (0,1)\big) \times B_\rho'(x') \subset \bR \times \bR^{d-1},
$$
where
$$
B_\rho'(x') := \{y' \in \bR^{d-1}: |y'-x'|< \rho\}.
$$
We see that $(x_1-h+\rho,z'), (x_1,z'), (x_1+h-\rho,z') \in B_1$ if $z \in B_\rho'(x')$.
Moreover,
$$
[x_1-h + \rho, x_1+h-\rho] \times B_\rho'(x') \subset B_1
$$
because if
$$
(z_1,z') \in [x_1-h + \rho, x_1+h-\rho] \times B_\rho'(x'),
$$
then
$$
|x_1-z_1| \leq h-\rho, \quad |x'-z'| < \rho,
$$
and
\begin{align*}
|(z_1,z')| &\le |(z_1,z') - (x_1,x')| + |x| \leq |x_1-z_1|+|x'-z'| + |x|\\
&< h-\rho + \rho + |x| \leq 1,
\end{align*}
where, for the last inequality, we used the assumption that $B_h(x) \subset B_1$.
For
$$
(s,z') \in \big((t-\rho^{2/\alpha}, t+\rho^{2/\alpha}) \cap (0,1)\big) \times B_\rho'(x'),
$$
we write
\begin{align}
							\label{eq0225_05}
J &\leq \left|\psi(s,x_1 + h - \rho,z') - 2 \psi(s,x_1,z') + \psi(s,x_1 - h+\rho, z')\right|\nonumber
\\
&\quad + \left|\psi(t,x_1 +h-\rho,x') - \psi(s,x_1 + h-\rho,z')\right|\nonumber
\\
&\quad + 2 \left|\psi(s,x_1,z') - \psi(t,x_1,x')\right|\nonumber
\\
&\quad + \left|\psi(t,x_1 - h + \rho, x') - \psi(s,x_1 - h + \rho, z')\right|\nonumber
\\
&\leq 4 K \rho^\sigma + \left|\psi(s,x_1 + h-\rho,z') - 2 \psi(s,x_1,z') + \psi(s,x_1 - h+\rho, z')\right|,
\end{align}
where
\begin{align*}
&\psi(s,x_1 + h-\rho,z') - 2 \psi(s,x_1,z') + \psi(s,x_1 - h+\rho, z')\\
&= \int_{x_1}^{x_1+h-\rho} \int_{2x_1-r}^r D_1^2 \psi(s,z_1,z') \, d z_1 \, dr.
\end{align*}
Hence, from this along with \eqref{eq0225_03}, \eqref{eq0225_04}, and \eqref{eq0225_05}, we obtain that
\begin{multline}
							\label{eq0225_01}
|\psi(t,x)-\psi(t,y)| \leq \frac{1}{2} (2h)^\sigma K + K \rho^\sigma + 2 K \rho^\sigma
\\
+ \frac{1}{2} \int_{x_1}^{x_1+h-\rho}\int_{2x_1-r}^r |D_1^2\psi(s,z_1,z')|\, dz_1 \, dr
\end{multline}
for any $(s,z')$ satisfying
$$
\big((t-\rho^{2/\alpha}, t+\rho^{2/\alpha}) \cap (0,1)\big) \times B_\rho'(x') =: \cD.
$$
By taking the average of both sides of \eqref{eq0225_01}
over the domain $\cD$
with respect to $(s,z')$
along with H\"{o}lder's inequality (note that $h-\rho > h/2$), we finally arrive at \eqref{eq0226_01}.
\end{proof}

\begin{lemma}[Embedding with $2$-spatial derivatives with $p\in(d/2+1/\alpha,d+2/\alpha)$]
							\label{lem0225_2}
Under the assumptions of Lemma \ref{lem0225_1}, for any $\varepsilon$ satisfying
\begin{equation}
							\label{eq0225_06}
0 < \varepsilon < (1-2^{\sigma-1})^{1/\sigma},
\end{equation}
we have
$$
K_1 \leq \frac{2^{1+\sigma}}{1-2^{\sigma-1}-\varepsilon^\sigma} M + \frac{2\varepsilon^\sigma}{1-2^{\sigma-1}-\varepsilon^\sigma}K + N \frac{\varepsilon^{-2/(\alpha p) - d/p + 1/p}}{1-2^{\sigma-1} - \varepsilon^\sigma}\|D^2_x \psi\|_{L_p\left((0,1) \times B_1 \right)},
$$
where $N = N(d,\alpha,p)$,
$$
M = \sup_{(t,x) \in (0,1) \times B_1} |\psi(t,x)|,
$$
$$
K_1 = \sup_{\substack{(t,x),(t,y) \in (0,1)\times B_1 \\ (t,x) \neq (t,y), y = \theta x, \theta \in \bR}}\frac{|\psi(t,x) - \psi(t,y)|}{|x-y|^\sigma},
$$
and $K$ is defined as in \eqref{eq0225_02}.
\end{lemma}

\begin{remark}
The quantity $K_1$ is the H\"{o}lder semi-norm of $\psi$ when $x$ and $y$ are on the same line passing through the origin.
\end{remark}

\begin{proof}
Thanks to an appropriate transformation, to estimate $K_1$, it is enough to estimate
$$
\frac{|\psi(t,x_1,0) - \psi(t,y_1,0)|}{|x_1-y_1|^\sigma}
$$
for $x_1, y_1 \in (-1,1)$.
For $x_1, y_1 \in (-1,1)$ such that $h:=|x_1 - y_1| \geq 1/2$, we see that
\begin{equation}
						\label{eq0224_06}
\frac{|\psi(t,x_1,0)-\psi(t,y_1,0)|}{|x_1-y_1|^\sigma} \leq 2^{1+\sigma} M.
\end{equation}
When $h < 1/2$, either $2x_1-y_1$ or $2y_1-x_1$ is in $(-1,1)$. Without loss of generality we assume that
$$
y_1 = x_1 - h,\quad x_1+h\in (-1,1).
$$
Set
$$
\rho = \varepsilon h,
$$
where $\varepsilon$ is a number satisfying \eqref{eq0225_06}.
Since
\begin{align*}
&\psi(t,x_1,0)-\psi(t,x_1-h,0)
= \frac{1}{2}\left[ \psi(t,x_1 + h,0) - \psi(t,x_1-h,0) \right]\\
&\quad - \frac{1}{2} \left[ \psi(t,x_1+h,0) - 2 \psi(t,x_1,0) + \psi(t,x_1-h,0) \right],
\end{align*}
we have
\begin{multline}
							\label{eq0224_02}
|\psi(t,x_1,0)-\psi(t,y_1,0)| \leq \frac{1}{2} (2h)^\sigma K_1
\\
+ \frac{1}{2} \left|\psi(t,x_1+h,0) - 2 \psi(t,x_1,0) + \psi(t,x_1-h,0)\right|.
\end{multline}
To estimate the last term in the above inequalities,
we observe that
\begin{align}
							\label{eq0224_03}
&\left|\psi(t,x_1+h,0) - 2 \psi(t,x_1,0) + \psi(t,x_1-h,0)\right|\nonumber
\\
&\leq |\psi(t,x_1+h-\rho,0)-2\psi(t,x_1,0)+\psi(t,x_1-h+\rho,0)|\nonumber
\\
&\quad + |\psi(t,x_1+h,0) - \psi(t,x_1+h-\rho,0)| + |\psi(t,x_1-h,0) - \psi(t,x_1-h+\rho,0)|\nonumber
\\
&\leq 2 K_1 \rho^\sigma + |\psi(t,x_1+h-\rho,0)-2\psi(t,x_1,0)+\psi(t,x_1-h+\rho,0)|.
\end{align}
We note that $(x_1-h+\rho,z'), (x_1,z'), (x_1+h-\rho,z') \in B_1$.
Moreover,
$$
[x_1-h + \rho, x_1+h-\rho] \times B_\rho'(0) \subset B_1,
$$
where $B_\rho'(0) = \{y' \in \bR^{d-1}: |y'|< \rho\}$.
To estimate the last term in \eqref{eq0224_03}, for
$$
(s,z') \in \big((t-\rho^{2/\alpha}, t+\rho^{2/\alpha}) \cap (0,1)\big) \times B_\rho'(0),
$$
we write
\begin{align}
							\label{eq0224_04}
&|\psi(t,x_1+h-\rho,0)-2\psi(t,x_1,0)+\psi(t,x_1-h+\rho,0)|\nonumber
\\
&\leq \left|\psi(s,x_1 + h - \rho,z') - 2 \psi(s,x_1,z') + \psi(s,x_1 - h+\rho, z')\right|\nonumber
\\
&\qquad+ \left|\psi(t,x_1 +h-\rho,0) - \psi(s,x_1 + h-\rho,z')\right|+ 2 \left|\psi(s,x_1,z') - \psi(t,x_1,0)\right|\nonumber
\\
&\qquad + \left|\psi(t,x_1 - h + \rho, 0) - \psi(s,x_1 - h + \rho, z')\right|\nonumber
\\
&\leq 4 K \rho^\sigma + \left|\psi(s,x_1 + h-\rho,z') - 2 \psi(s,x_1,z') + \psi(s,x_1 - h+\rho, z')\right|.
\end{align}
Note that
\begin{align*}
&\psi(s,x_1 + h-\rho,z') - 2 \psi(s,x_1,z') + \psi(s,x_1 - h+\rho, z')\\
&= \int_{x_1}^{x_1+h-\rho} \int_{2x_1-r}^r D_1^2 u(s,z_1,z') \, d z_1 \, dr.
\end{align*}
Hence, from this along with \eqref{eq0224_02}, \eqref{eq0224_03}, and \eqref{eq0224_04}, we obtain that
\begin{multline}
							\label{eq0224_05}
|\psi(t,x_1,0)-\psi(t,y_1,0)| \leq \frac{1}{2} (2h)^\sigma K_1 + K_1 \rho^\sigma + 2 K \rho^\sigma
\\
+ \frac{1}{2} \int_{x_1}^{x_1+h-\rho}\int_{2x_1-r}^r |D_1^2u(s,z_1,z')|\, dz_1 \, dr
\end{multline}
for any $(s,z')$ satisfying
$$
\big((t-\rho^{2/\alpha}, t+\rho^{2/\alpha}) \cap (0,1)\big) \times B_\rho'(0) =: \cD.
$$
By taking the average of both sides of \eqref{eq0224_05}
over the domain $\cD$
with respect to $(s,z')$
along with H\"{o}lder's inequality (note that $h-\rho > h/2$), we arrive at
\begin{align*}
&|\psi(t,x_1,0)-\psi(t,y_1,0)|
 \leq (2^{\sigma-1} h^\sigma + \varepsilon^\sigma h^\sigma) K_1 + 2 \varepsilon^\sigma h^\sigma K\\
&\qquad+ N \varepsilon^{-2/(\alpha p) - d/p + 1/p} h^\sigma \|D_x^2\psi \|_{L_p\left((0,1) \times B_1 \right)}
\end{align*}
whenever $h = |x_1-h_1| < 1/2$.
From this and \eqref{eq0224_06}, we conclude that
$$
K_1 \leq 2^{1+\sigma}M
+ (2^{\sigma-1} + \varepsilon^\sigma) K_1 + 2 \varepsilon^\sigma K + N \varepsilon^{-2/(\alpha p) - d/p + 1/p} \|D_x^2\psi \|_{L_p\left((0,1) \times B_1 \right)}
$$
for any $\varepsilon$ satisfying \eqref{eq0225_06}, where $N = N(d,\alpha,p)$.
This shows that
\begin{align*}
K_1 &\leq \frac{2^{1+\sigma}}{1-2^{\sigma-1}-\varepsilon^\sigma} M + \frac{2\varepsilon^\sigma}{1-2^{\sigma-1}-\varepsilon^\sigma}K\\
&\quad + N \frac{\varepsilon^{-2/(\alpha p) - d/p + 1/p}}{1-2^{\sigma-1} - \varepsilon^\sigma}\|D^2_x \psi\|_{L_p\left((0,1) \times B_1 \right)}.
\end{align*}
The lemma is proved.
\end{proof}

\begin{theorem}[Embedding with $\alpha$-time derivative and $2$-spatial derivatives with $p\in(d/2+1/\alpha,d+2/\alpha)$]
                                        \label{thm5.18}
Let $\alpha \in (0,1)$ and $p \in (1,\infty)$ such that
$$
\sigma := 2-(d+2/\alpha)/p \in (0,1).
$$
Then, for $\bH_{p,0}^{\alpha,2}\left((0,1) \times B_1\right)$, we have
\begin{equation}
								\label{eq0224_01}
[\psi]_{C^{\sigma \alpha/2, \sigma}\left((0,1) \times B_1\right)} \leq N(d,\alpha,p) \|\psi\|_{\bH_p^{\alpha,2}\left((0,1) \times B_1\right)}.
\end{equation}
\end{theorem}

\begin{proof}
By the definition of $\bH_{p,0}^{\alpha,2}\left((0,1) \times B_1\right)$ and Remark \ref{rem0606_1}, we may assume that $\psi \in C^\infty_0 \big(\overline{(0,1)\times B_1}\big)$ and $\psi(0,x) = 0$.
To prove \eqref{eq0224_01}, we take $(t_1,x), (t_2,y) \in (0,1) \times B_1$, $(t_1,x) \neq (t_2,y)$, and set
$$
\rho = \varepsilon \left( |t_1-t_2|^{\alpha/2} + |x-y| \right),
$$
where $\varepsilon \in (0,1)$ is to be specified below.
We write
$$
|\psi(t_1,x) - \psi(t_2,y)| \leq |\psi(t_1,x) - \psi(t_2,x)| + |\psi(t_2,x) - \psi(t_2,y)| := J_1 + J_2.
$$
To estimate $J_1$, for $z \in B_\rho(x) \cap B_1$, we have
\begin{align*}
J_1 &\leq |\psi(t_1,x)-\psi(t_1,z)| + |\psi(t_1,z) - \psi(t_2,z)| + |\psi(t_2,z) -\psi(t_2,x)|\\
&\leq 2 K \rho^\sigma + |\psi(t_1,z) - \psi(t_2,z)|,
\end{align*}
where by
Lemma \ref{lem0217_2} we see that
$$
|\psi(t_1,z) - \psi(t_2,z)| \le N(\alpha,p)|t_1-t_2|^{\alpha-1/p} \|\partial_t^\alpha \psi(\cdot,z)\|_{L_p(0,1)}.
$$
Then by taking the average of $J_1$ over $B_\rho(x) \cap B_1$ with respect to $z$ along with H\"{o}lder's inequality (note that $|B_\rho(x) \cap B_1| \ge N(d) |B_\rho(x)|$), we get
\begin{align*}
J_1 &\leq 2K \rho^\sigma + N |t_1-t_2|^{\alpha-1/p} \rho^{-d/p} \| \partial_t^\alpha \psi\|_{L_p\left((0,1) \times B_1\right)}\\
&\leq 2 K \rho^\sigma + N \varepsilon^{-2+2/(\alpha p)} \rho^\sigma \| \partial_t^\alpha \psi\|_{L_p\left((0,1) \times B_1\right)},
\end{align*}
where $N = N(d,\alpha,p)$.

We now estimate $J_2$.
First, recall the definitions of $M$, $K$, and $K_1$ from Lemmas \ref{lem0225_1} and \ref{lem0225_2}.
If $|x-y| \ge 1/8$, we have
$$
\frac{|\psi(t_2,x)-\psi(t_2,y)|}{|x-y|^\sigma} \leq 2 \cdot 8^\sigma M.
$$
Assume that $|x-y| =: h < 1/8$.
If $B_h(x) \subset B_1$ or $B_h(y) \subset B_1$, by Lemma \ref{lem0225_1} we have
$$
\frac{|\psi(t,x)-\psi(t,y)|}{|x-y|^\sigma}
 \leq (2^{\sigma-1} + 3 \varepsilon_1^\sigma) K + N \varepsilon_1^{-2/(\alpha p) - d/p + 1/p} \|D_x^2\psi \|_{L_p\left((0,1) \times B_1 \right)}
$$
for any $\varepsilon_1 \in (0,1/2)$.

Now we consider the case that $x, y \in B_1$, $h:=|x-y| < 1/8$, and
$$
B_h(x) \not\subset B_1 \quad \text{and} \quad B_h(y) \not\subset B_1.
$$
Without loss of generality, we assume that $|y| \ge |x|$.
Then we see that
$$
|y| \ge 7/8, \quad |x| \geq 7/8, \quad |y| - h > 0.
$$
Set
$$
\tilde{y} = \frac{|y|-h}{|y|}y, \quad \tilde{x} = \frac{|y|-h}{|y|}x.
$$
Then
$$
|y - \tilde{y}| = h, \quad |x - \tilde{x}| = h |x|/|y| \leq h,
$$
$$
|\tilde{x}-\tilde{y}| = (1-h/|y|)h =: \tilde{h} < h.
$$
Moreover,
$$
B_{\tilde{h}}(\tilde{y}) \subset B_1
$$
because, for any $z \in B_{\tilde{h}}(\tilde{y})$,
$$
|z| \leq |z-\tilde{y}| + |\tilde{y}| < \tilde{h} + |y| - h < 1.
$$
We observe that
\begin{align*}
&|h|^{-\sigma} J_2 = \frac{|\psi(t_2,x) - \psi(t_2,y)|}{|h|^\sigma}\\
&\leq \frac{|\psi(t_2,x) - \psi(t_2,\tilde{x})|}{|h|^\sigma} + \frac{|\psi(t_2,\tilde{x}) - \psi(t_2,\tilde{y})|}{|h|^\sigma} + \frac{|\psi(t_2,\tilde{y}) - \psi(t_2,y)|}{|h|^\sigma}\\
&\leq \frac{|\psi(t_2,x) - \psi(t_2,\tilde{x})|}{|x-\tilde{x}|^\sigma} + \frac{|\psi(t_2,\tilde{x}) - \psi(t_2,\tilde{y})|}{|\tilde{h}|^\sigma} + \frac{|\psi(t_2,\tilde{y}) - \psi(t_2,y)|}{|h|^\sigma}\\
&=: J_{2,1} + J_{2,2} + J_{2,3},
\end{align*}
where we note that $x$ and $\tilde{x}$ are on the same line passing through the origin, so do $y$ and $\tilde{y}$.
Thus, by Lemma \ref{lem0225_2},
\begin{align*}
&J_{2,1} + J_{2,3}\\
&\leq \frac{2^{2+\sigma}}{1-2^{\sigma-1}-\varepsilon_2^\sigma} M + \frac{4\varepsilon_2^\sigma}{1-2^{\sigma-1}-\varepsilon_2^\sigma}K + N \frac{\varepsilon_2^{-2/(\alpha p) - d/p + 1/p}}{1-2^{\sigma-1} - \varepsilon_2^\sigma}\|D^2_x \psi\|_{L_p\left((0,1) \times B_1 \right)}
\end{align*}
for any $\varepsilon_2$ satisfying \eqref{eq0225_06}.
For $J_{2,2}$, since $B_{\tilde{h}}(\tilde{y}) \subset B_1$, by Lemma \ref{lem0225_1}, we obtain that
$$
J_{2,2} \leq (2^{\sigma-1}  + 3 \varepsilon_3^\sigma) K + N \varepsilon_3^{-2/(\alpha p) - d/p + 1/p} \|D_x^2\psi \|_{L_p\left((0,1) \times B_1 \right)}
$$
for any $\varepsilon_3 \in (0,1/2)$.

We collect the estimates for $J_1$ and $J_2$ along with those for $J_{2,1}$, $J_{2,2}$, and $J_{2,3}$ as follows.
Set
$$
J = \frac{|\psi(t_1,x) - \psi(t_2,y)|}{|t_1-t_2|^{\sigma \alpha/2} + |x-y|^\sigma}.
$$
If $|x-y| \geq 1/8$, then
$$
J \leq 2K \varepsilon^\sigma + N \varepsilon^{-d/p} \|\partial_t^\alpha \psi\|_{L_p\left((0,1) \times B_1\right)} + 2 \cdot 8^\sigma M
$$
for $\varepsilon \in (0,1)$.
If $h := |x-y| < 1/8$ and $B_h(x) \subset B_1$ or $B_h(y) \subset B_1$, then
\begin{align*}
&J \leq 2K \varepsilon^\sigma + N \varepsilon^{-d/p} \|\partial_t^\alpha \psi\|_{L_p\left((0,1) \times B_1\right)}\\
&\quad + (2^{\sigma-1} + 3 \varepsilon_1^\sigma) K + N \varepsilon_1^{-2/(\alpha p) - d/p + 1/p} \|D_x^2\psi \|_{L_p\left((0,1) \times B_1 \right)}
\end{align*}
for $\varepsilon \in (0,1)$ and $\varepsilon_1 \in (0,1/2)$.
If $h=|x-y| < 1/8$, and $B_h(x) \not\subset B_1$ and $B_h(y) \not\subset B_1$, then
\begin{align*}
J &\leq 2K \varepsilon^\sigma + N \varepsilon^{-d/p} \|\partial_t^\alpha \psi\|_{L_p\left((0,1) \times B_1\right)}\\
&\,\, + \frac{2^{2+\sigma}}{1-2^{\sigma-1}-\varepsilon_2^\sigma} M + \frac{4\varepsilon_2^\sigma}{1-2^{\sigma-1}-\varepsilon_2^\sigma}K + N \frac{\varepsilon_2^{-2/(\alpha p) - d/p + 1/p}}{1-2^{\sigma-1} - \varepsilon_2^\sigma}\|D^2_x \psi\|_{L_p\left((0,1) \times B_1 \right)}\\
&\,\,+ (2^{\sigma-1}  + 3 \varepsilon_3^\sigma) K + N \varepsilon_3^{-2/(\alpha p) - d/p + 1/p} \|D_x^2\psi \|_{L_p\left((0,1) \times B_1 \right)}
\end{align*}
for $\varepsilon \in (0,1)$, $\varepsilon_2$ satisfying \eqref{eq0225_06}, and $\varepsilon_3 \in (0,1/2)$.
The above three inequalities show that
\begin{align*}
J &\leq 2K \varepsilon^\sigma + N \varepsilon^{-d/p} \|\partial_t^\alpha \psi\|_{L_p\left((0,1) \times B_1\right)} + 2 \cdot 8^\sigma M\\
&\,\, + 2^{\sigma-1}K + 3 (\varepsilon_1^\sigma + \varepsilon_3^\sigma)K\\
&\,\, + N (\varepsilon_1^{-2/(\alpha p) - d/p + 1/p} + \varepsilon_3^{-2/(\alpha p) - d/p + 1/p})\|D_x^2\psi \|_{L_p\left((0,1) \times B_1 \right)}\\
&\,\, + \frac{2^{2+\sigma}}{1-2^{\sigma-1}-\varepsilon_2^\sigma} M + \frac{4\varepsilon_2^\sigma}{1-2^{\sigma-1}-\varepsilon_2^\sigma}K + N \frac{\varepsilon_2^{-2/(\alpha p) - d/p + 1/p}}{1-2^{\sigma-1} - \varepsilon_2^\sigma}\|D^2_x \psi\|_{L_p\left((0,1) \times B_1 \right)}
\end{align*}
for any $\varepsilon \in (0,1)$, $\varepsilon_1, \varepsilon_3 \in (0,1/2)$, and $\varepsilon_2$ satisfying \eqref{eq0225_06}, where it is crucial that there is only one term $2^{\sigma-1}K$ on the right-hand side of the inequality.
Since $(t_1,x)$ and $(t_2,y)$ are arbitrary points in $(0,1) \times B_1$, we see that
\begin{align*}
K &\leq 2K \varepsilon^\sigma + N \varepsilon^{-d/p} \|\partial_t^\alpha \psi\|_{L_p\left((0,1) \times B_1\right)} + 2 \cdot 8^\sigma M\\
&\,\, + 2^{\sigma-1}K + 3 (\varepsilon_1^\sigma + \varepsilon_3^\sigma)K\\
&\,\, + N (\varepsilon_1^{-2/(\alpha p) - d/p + 1/p} + \varepsilon_3^{-2/(\alpha p) - d/p + 1/p})\|D_x^2\psi \|_{L_p\left((0,1) \times B_1 \right)}\\
&\,\, + \frac{2^{2+\sigma}}{1-2^{\sigma-1}-\varepsilon_2^\sigma} M + \frac{4\varepsilon_2^\sigma}{1-2^{\sigma-1}-\varepsilon_2^\sigma}K + N \frac{\varepsilon_2^{-2/(\alpha p) - d/p + 1/p}}{1-2^{\sigma-1} - \varepsilon_2^\sigma}\|D^2_x \psi\|_{L_p\left((0,1) \times B_1 \right)}.
\end{align*}
Upon using the fact that  $2^{\sigma-1} < 1$, we fix $\varepsilon$, $\varepsilon_1$, $\varepsilon_2$, and $\varepsilon_3$ small enough depending on $d$, $\alpha$, and $p$ so that
$$
1-2^{\sigma-1} - 2\varepsilon^\sigma - 3(\varepsilon_1^\sigma+\varepsilon_3^\sigma) - \frac{4 \varepsilon_2^\sigma}{1-2^{\sigma-1}-\varepsilon_2^\sigma} > 0.
$$
Then
$$
K \leq N M + N \|\partial_t^\alpha \psi\|_{L_p\left((0,1) \times B_1\right)} + N \|D_x^2\psi \|_{L_p\left((0,1) \times B_1 \right)}.
$$
Finally, we observe that, by interpolation inequalities, for any $\varepsilon_4 > 0$,
\begin{equation}
                                        \label{eq2.36}
M \leq \varepsilon_4 K + N(d,\alpha,p,\varepsilon_4) \|\psi\|_{L_p\left((0,1) \times B_1\right)}.
\end{equation}
The theorem is proved.
\end{proof}

\begin{corollary}[Embedding with $\alpha$-time derivative and $2$-spatial derivatives with $p \in (d/2+1/\alpha,d+2/\alpha)$]
Let $\alpha \in (0,1)$ and $p \in (1,\infty)$ such that
$$
\sigma := 2-(d+2/\alpha)/p \in (0,1).
$$
Then, for $\psi \in \bH_{p,0}^{\alpha,2}\left((0,1) \times \bR^d\right)$, we have
$$
[\psi]_{C^{\sigma \alpha/2, \sigma}\left((0,1) \times \bR^d\right)} \leq N(d,\alpha,p) \|\psi\|_{\bH_p^{\alpha,2}\left((0,1) \times \bR^d\right)}.
$$
\end{corollary}

\begin{proof}
As in the proof Theorem \ref{thm1204_1}, we assume that $\psi \in C^\infty_0 \left([0,1] \times \bR^d\right)$ and $\psi(0,x) = 0$.
Consider
$$
\frac{|\psi(t,x) - \psi(s,y)|}{|t-s|^{\sigma \alpha/2} + |x-y|^\sigma}
$$
for two different points $(t,x),(s,y) \in (0,1)\times \bR^d$. If $|x-y|<1$, then we apply Theorem \ref{thm5.18} with a shift of the coordinates. If $|x-y|>1$, then the above quantity is bounded by $2\|\psi\|_{L_\infty((0,1)\times \bR^d)}$, and it suffices to use the interpolation inequality \eqref{eq2.36}.
\end{proof}

The following is a version of the ``crawling of ink spots'' lemma to be used in the proofs of the main results of this paper.
Note that the underlying set $(-\infty,T) \times \bR^d$ is unbounded.
Recall the definitions of $\cC_R(t,x)$ and $\hat{\cC}_R(t,x)$ in \eqref{eq0606_01}.

\begin{lemma}
							\label{lem0409_2}
Let $\gamma \in (0,1)$ and $|E| < \infty$.
Suppose that $E \subset F \subset (-\infty,T) \times \bR^d$ and, for any $(t,x) \in (-\infty,T] \times \bR^d$ and for all $R \in (0,\infty)$ with
$$
\left| \cC_R(t,x) \cap E \right| \ge \gamma |\cC_R(t,x)|,
$$
we have
$$
 \hat\cC_R(t,x) \subset F.
$$
Then
\begin{equation}
							\label{eq0409_03}
|E| \leq N(d,\alpha) \gamma |F|.
\end{equation}
\end{lemma}

\begin{proof}
For every $(t,x) \in E$, define
$$
\varphi_{(t,x)}(r) := \frac{|E \cap \cC_r(t,x)|}{|\cC_r(t,x)|} \leq \frac{|E|}{|\cC_r(t,x)|} \to 0
$$
as $r \to \infty$.
On the other hand, by the Lebesgue differentiation theorem, for almost every $(t,x) \in E$,
$$
\lim_{r \to 0} \varphi_{(t,x)}(r) = 1.
$$
Moreover, $\varphi_{(t,x)}(r)$ is continuous on $(0,\infty)$.
Since $\gamma \in (0,1)$, for almost every $(t,x) \in E$, there exits $r \in (0,\infty)$ such that
$$
\varphi_{(t,x)}(r) = \gamma.
$$
Then we set
$$
R = R(t,x) = \sup\{ r \in (0,\infty): \varphi_{(t,x)}(r) = \gamma\},
$$
where we understand that $\inf \emptyset = \infty$.
Then $0 < R(t,x) \leq
\infty$.
Define
$$
\Gamma_1 = \{\cC = \cC_{R(t,x)}(t,x): (t,x) \in E, \,\, R(t,x) < \infty\}
$$
and
$$
R^*_1 = \sup \{R(t,x): \cC_{R(t,x)}(t,x) \in \Gamma_1\}.
$$
Note that
$$
E \setminus N \subset \bigcup_{\cC_{R(t,x)}(t,x) \in \Gamma_1} \cC_{R(t,x)}(t,x),
$$
where $N$ is a null set.

If $R^*_1 = \infty$, then $\Gamma_1$ contains a sequence of $\cC_{R_k}(t_k,x_k) := \cC_{R(t_k,x_k)}(t_k,x_k)$ with
$$
|\cC_{R_k}(t_k,x_k)| \to \infty
$$
as $k \to \infty$.
In this case, choose $k_1 \in \bN$ such that
$$
|\cC_{R_{k_1}}(t_{k_1},x_{k_1})| \geq 2 \gamma^{-1}|E|.
$$
Since
$$
|\cC_{R_{k_1}}(t_{k_1},x_{k_1}) \cap E | = \gamma |\cC_{R_{k_1}}(t_{k_1},x_{k_1})|,
$$
by the assumption in the lemma, we have
$$
\hat\cC_{R_{k_1}}(t_{k_1},x_{k_1}) \subset F.
$$
It then follows that
$$
|E| \leq \frac{\gamma}{2} |\cC_{R_{k_1}}(t_{k_1},x_{k_1})| \leq  \gamma |\hat\cC_{R_{k_1}}(t_{k_1},x_{k_1})| \leq \gamma |F|.
$$
Hence, we obtain \eqref{eq0409_03}.

If $R^*_1 < \infty$, we find a countable sub-collection $\Gamma_0$ of $\Gamma_1$ as follows.
Choose $\cC_{R_1}(t_1,x_1):=\cC_{R(t_1,x_1)}(t_1,x_1)$ from $\Gamma_1$ such that $R_1 > R^*_1/2$.
Now spit $\Gamma_1 = \Gamma_2 \cup \Gamma_2'$, where $\Gamma_2$ consists of those $\cC_{R(t,x)}(t,x)$ disjoint from $\cC_{R_1}(t_1,x_1)$, and $\Gamma_2'$ of those which intersect $\cC_{R_1}(t_1,x_1)$.
Now we note that
$$
\cC_{R(t,x)}(t,x) \subset \cC_{5R_1}(t_1,x_1),
$$
whenever $\cC_{R(t,x)}(t,x) \in \Gamma_2'$.

Now assume that $\cC_{R_k}(t_k,x_k)$ and $\Gamma_{k+1}$ are chosen.
If $\Gamma_{k+1}$ is empty, the process ends.
If not, we choose $\cC_{R_{k+1}}(t_{k+1}, x_{k+1})$ from $\Gamma_{k+1}$ such that
$$
R_{k+1} > \frac{1}{2} R^*_{k+1}, \quad
R^*_{k+1}:= \sup_{\cC_{R(t,x)}(t,x) \in \Gamma_{k+1}} R(t,x).
$$
Then split $\Gamma_{k+1} = \Gamma_{k+2} \cup \Gamma_{k+2}'$, where $\Gamma_{k+2}$ consists of those $\cC_{R(t,x)}(t,x)$ disjoint from $\cC_{R_{k+1}}(t_{k+1}, x_{k+1})$, and $\Gamma_{k+2}'$ of those which intersect $\cC_{R_{k+1}}(t_{k+1}, x_{k+1})$.
Now we set
$$
\Gamma_0 = \{\cC_{R_k}(t_k,x_k) : k = 1, 2, \ldots \}.
$$
Clearly, we have $\cC_{R_k}(t_k,x_k) \cap \cC_{R_j}(t_j,x_j) = \emptyset$ if $k \neq j$.

Now we prove \eqref{eq0409_03} when $R^*_1 < \infty$.
First, consider the case that $\Gamma_0$ contains only finitely many elements or $\Gamma_0$ has infinitely many elements with $R_k^* \searrow 0$.
Then
\begin{equation}
							\label{eq0409_05}
\Gamma_1 = \bigcup_{k=2}^\infty \Gamma_k'.
\end{equation}
In particular, when $R_k^* \searrow 0$, if there exits $\cC_{R(t,x)}(t,x) \in \Gamma_1$ such that
$$
\cC_{R(t,x)}(t,x) \not\in \bigcup_{k=2}^\infty \Gamma_k',
$$
then $\cC_{R(t,x)}(t,x) \in \Gamma_k$ for all $k=1,2,\ldots$.
This means that $R(t,x) = 0$, which is a contradiction because $R(t,x) > 0$.
From \eqref{eq0409_05} and the fact that
$$
\cC_{R(t,x)}(t,x) \subset \cC_{5 R_k(t_k,x_k)}
$$
for any $\cC_{R(t,x)}(t,x) \in \Gamma_{k+1}'$ and $k = 1,2,\ldots$, we have
\begin{equation}
								\label{eq0409_04}
E \setminus N \subset \bigcup_{\cC_{R(t,x)}(t,x) \in \Gamma_1} \cC_{R(t,x)}(t,x) \subset \bigcup_{k=1}^\infty \cC_{5R_k}(t_k,x_k),
\end{equation}
where $N$ is a null set.
Note that, for each $k = 1,2,\ldots$,
$$
|\cC_{R_k}(t_k,x_k) \cap
E| = \gamma|\cC_{R_k}(t_k,x_k)|,
$$
$$
|\cC_{5R_k}(t_k,x_k) \cap
E| < \gamma|\cC_{5R_k}(t_k,x_k)|.
$$
Hence, by the assumption,
$$
\hat\cC_{R_k}(t_k,x_k) \subset F,
$$
and by the disjointness of $\cC_{R_k}(t_k,x_k)$ and \eqref{eq0409_04} we have
\begin{align*}
&|E| \leq | \bigcup_{k=1}^\infty E \cap \cC_{5R_k}(t_k,x_k)| \leq \sum_{k=1}^\infty |E \cap \cC_{5R_k}(t_k,x_k)|\\
&\leq \gamma \sum_{k=1}^\infty |\cC_{5R_k}(t_k,x_k)| = 5^{d + 2/\alpha} \gamma \sum_{k=1}^\infty |\cC_{R_k}(t_k,x_k)|  = 5^{d + 2/\alpha} \gamma \left|\bigcup_{k=1}^\infty \cC_{R_k}(t_k,x_k)\right|\\
&\leq 5^{d+2/\alpha} 2 \gamma \left| \bigcup_{k=1}^\infty \hat\cC_{R_k}(t_k,x_k)\right| \leq N(d,\alpha) \gamma |F|.
\end{align*}
Thus we obtain \eqref{eq0409_03}.
Now assume that there exists a number $\varepsilon_0 > 0$ such that $R_k^* \ge \varepsilon_0$ for all $k = 1,2, \ldots$.
This means that
$$
\left|\bigcup_{k=1}^M \cC_{R_k}(t_k,x_k)\right| = \sum_{k=1}^M |\cC_{R_k}(t_k,x_k)| \to \infty
$$
as $M \to \infty$.
Then we find $M$ such that
$$
\left|\bigcup_{k=1}^M \cC_{R_k}(t_k,x_k)\right| \geq 2 \gamma^{-1}|E|.
$$
Since $\hat\cC_{R_k}(t_k,x_k) \subset F$, we have
$$
|E| \leq \frac{\gamma}{2} \sum_{k=1}^M |\cC_{R_k}(t_k,x_k)|  \leq \gamma \left| \bigcup_{k=1}^M \hat\cC_{R_k}(t_k,x_k) \right| \leq \gamma |F|.
$$
Thus we again arrive at \eqref{eq0409_03}.
\end{proof}

\bibliographystyle{plain}

\begin{thebibliography}{10}

\bibitem{MR3488533}
Mark Allen, Luis Caffarelli, and Alexis Vasseur.
\newblock A parabolic problem with a fractional time derivative.
\newblock {\em Arch. Ration. Mech. Anal.}, 221(2):603--630, 2016.

\bibitem{MR1486629}
L.~A. Caffarelli and I.~Peral.
\newblock On {$W^{1,p}$} estimates for elliptic equations in divergence form.
\newblock {\em Comm. Pure Appl. Math.}, 51(1):1--21, 1998.

\bibitem{DK15}
Hongjie Dong and Doyoon Kim.
\newblock On {$L_p$}-estimates for elliptic and parabolic equations with
  {$A_p$} weights.
\newblock {\em Trans. Amer. Math. Soc.} 370(7):5081--5130, 2018.

\bibitem{MR2771670}
Hongjie Dong and Doyoon Kim.
\newblock On the {$L_p$}-solvability of higher order parabolic and elliptic
  systems with {BMO} coefficients.
\newblock {\em Arch. Ration. Mech. Anal.}, 199(3):889--941, 2011.

\bibitem{MR1544927}
G.~H. Hardy and J.~E. Littlewood.
\newblock Some properties of fractional integrals. {I}.
\newblock {\em Math. Z.}, 27(1):565--606, 1928.

\bibitem{MR3581300}
Ildoo Kim, Kyeong-Hun Kim, and Sungbin Lim.
\newblock An {$L_q(L_p)$}-theory for the time fractional evolution equations
  with variable coefficients.
\newblock {\em Adv. Math.}, 306:123--176, 2017.

\bibitem{MR2304157}
N.~V. Krylov.
\newblock Parabolic and elliptic equations with {VMO} coefficients.
\newblock {\em Comm. Partial Differential Equations}, 32(1-3):453--475, 2007.

\bibitem{MR2352490}
N.~V. Krylov.
\newblock Parabolic equations with {VMO} coefficients in {S}obolev spaces with
  mixed norms.
\newblock {\em J. Funct. Anal.}, 250(2):521--558, 2007.

\bibitem{MR3488249}
N.~V. Krylov.
\newblock On parabolic equations in one space dimension.
\newblock {\em Comm. Partial Differential Equations}, 41(4):644--664, 2016.

\bibitem{MR563790}
N.~V. Krylov and M.~V. Safonov.
\newblock A property of the solutions of parabolic equations with measurable
  coefficients.
\newblock {\em Izv. Akad. Nauk SSSR Ser. Mat.}, 44(1):161--175, 239, 1980.

\bibitem{MR0164252}
Norman~G. Meyers and James Serrin.
\newblock {$H=W$}.
\newblock {\em Proc. Nat. Acad. Sci. U.S.A.}, 51:1055--1056, 1964.

\bibitem{MR579490}
M.~V. Safonov.
\newblock Harnack's inequality for elliptic equations and {H}\"older property
  of their solutions.
\newblock {\em Zap. Nauchn. Sem. Leningrad. Otdel. Mat. Inst. Steklov. (LOMI)},
  96:272--287, 312, 1980.
\newblock Boundary value problems of mathematical physics and related questions
  in the theory of functions, 12.

\bibitem{MR2538276}
Rico Zacher.
\newblock Weak solutions of abstract evolutionary integro-differential
  equations in {H}ilbert spaces.
\newblock {\em Funkcial. Ekvac.}, 52(1):1--18, 2009.

\bibitem{MR3038123}
Rico Zacher.
\newblock A {D}e {G}iorgi--{N}ash type theorem for time fractional diffusion
  equations.
\newblock {\em Math. Ann.}, 356(1):99--146, 2013.

\end{thebibliography}

\def\cprime{$'$}

\end{document}